\documentclass[a4paper,10pt,twoside]{article}
\usepackage[utf8]{inputenc} 
\usepackage{graphicx}
\usepackage[pdftex,bookmarks]{hyperref}
\usepackage{amsmath}    
\usepackage{amssymb}
\usepackage{amsmath,amsthm}
\usepackage{textcomp}
\usepackage{color}
\usepackage[all]{xy}
\usepackage{slashed}
\usepackage{latexsym}
\usepackage{bm,fixmath} 
\usepackage{mathrsfs}

\usepackage[titletoc,toc,title]{appendix}

\usepackage{fancyhdr}
\numberwithin{equation}{section}
\newtheorem{theorem}{Theorem}[section]
\newtheorem{lemma}[theorem]{Lemma}
\newtheorem{proposition}[theorem]{Proposition}

\theoremstyle{definition}
\newtheorem{definition}[theorem]{Definition}
\newtheorem{example}[theorem]{Exemple}

\newtheorem{remark}[theorem]{Remark}
\theoremstyle{plain}

\def\KK{\ensuremath\mathbb{K}}

\def\BB{\ensuremath\mathbb{B}}

\def\CC{\ensuremath{\mathbb{C}}}
\def\NN{\ensuremath{\mathbb{N}}}

\def\RR{\ensuremath{\mathbb{R}}}
\def\ZZ{\ensuremath{\mathbb{Z}}}
\date{}
\begin{document}
\pagestyle{plain}
	\title{The adiabatic groupoid and the Higson-Roe exact sequence}
	
	\author
	{Vito Felice Zenobi}
	\maketitle

		\begin{abstract}
Let $\widetilde{X}$ be a smooth Riemannian manifold equipped with a proper, free, isometric and cocompact action of a discrete group $\Gamma$.
In this paper we prove that the analytic surgery exact sequence of Higson-Roe for $\widetilde{X}$  is isomorphic to the exact sequence associated to the adiabatic deformation of the Lie groupoid $\widetilde{X}\times_\Gamma\widetilde{X}$. We then generalize this result to the context of smoothly stratified manifolds.
Finally, we show, by means of the aforementioned isomorphism, that the $\varrho$-classes associated to a metric with positive scalar curvature defined in \cite{PS1} corresponds to the $\varrho$-classes defined in \cite{Zadiabatic}.
		\end{abstract}
		
\section*{Introduction}

Let $\widetilde{X}$ be a proper metric space equipped with a  proper and co-compact action of a  discrete group $\Gamma$. 
 In \cite{Roeassembly} Roe  proved that the assembly map can be realized as the boundary map in K-theory associated to the short  exact sequence of C*-algebras 
\begin{equation}\label{1}
 \xymatrix{0\ar[r] & C^*(\widetilde{X})^\Gamma\ar[r]&  D^*(\widetilde{X})^\Gamma\ar[r] &  D^*(\widetilde{X})^\Gamma /C^*(\widetilde{X})^\Gamma\ar[r] & 0}
\end{equation}
we will call it the \emph{coarse} assembly map.
In their seminal papers \cite{HR1,HR2,HR3},  Higson and  Roe constructed a map from the surgery exact sequence  of Browder, Novikov, Sullivan and Wall
\begin{equation}\label{2}
\xymatrix{\cdots\ar[r] & L_*(\ZZ\Gamma)\ar[r]& S_*(X)\ar[r] &  \mathcal{N}_*(X)\ar[r] & \cdots}
\end{equation}
to 
 \begin{equation}\label{3}
 \xymatrix{\cdots\ar[r] & K_*(C^*(\widetilde{X})^\Gamma)\ar[r]&  K_*(D^*(\widetilde{X})^\Gamma)\ar[r] &  K_*(D^*(\widetilde{X})^\Gamma /C^*(\widetilde{X})^\Gamma)\ar[r] & \cdots}
 \end{equation}
which was called the \emph{analytic surgery exact sequence}, in analogy with its topological counterpart \eqref{2}.

In \cite{PS1} Piazza and Schick use index theoretic techniques  to map the Stolz' positive scalar curvature sequence to \eqref{3}. In \cite{PS2} they then revisit the mapping from \eqref{2} to \eqref{3}.
The main results of those papers are the definition of certain K-theoretic secondary invariants and the proof of the \emph{delocalized APS index theorem}.

The papers of Higson and Roe stimulated a fervent activity resulting in a number of different realizations of the analytic surgery exact sequence. In with follows we list a few of the main contributions. 
In \cite{zenobi}, the author of the present paper  uses Lipschitz structures  to generalize the results of \cite{PS2} from the setting of smooth manifolds to the one of topological manifolds. In the same paper a new exact sequence is introduced, isomorphic to \eqref{3}.  This new realization was then used for proving product formulas for secondary invariants. The group $\mathcal{S}_*^\Gamma(\widetilde{X})$, which corresponds to $K_*(D^*(\widetilde{X})^\Gamma)$, is given roughly speaking by the homotopy fiber of the Kasparov assembly map. Let us point out that if $\Gamma$ is a topological groupoid acting on a topological space $X$ and $A$ is a $\Gamma$-algebra, one also has a more general definition of $\mathcal{S}_*^\Gamma(\widetilde{X})$ which fits into the following exact sequence
\begin{equation}\label{4}
\xymatrix{\cdots\ar[r] & KK_*(\CC,C_0(0,1)\otimes A\rtimes\Gamma)\ar[r]& \mathcal{S}_*^\Gamma(\widetilde{X};A) \ar[r] & KK_*^\Gamma(C_0(\widetilde{X}), A)\ar[r]&\cdots}
\end{equation}
involving the assembly map for groupoid action with coefficients in the C*-algebra $A$, see \cite{Legall, JLT}.
In the two recent works  \cite{BR1, BR2}, Benameur and Roy  introduce the Higson-Roe exact sequence for the action of a transformation groupoid. 

In \cite{Yu} Yu introduces the so-called \emph{localization algebras} and an other assembly-type map (the \emph{local index map}), which is proved to be equivalent to the coarse assembly map, see also \cite{QR}.
Using the localization algebras many authors have contributed to the study of K-theoretic analytic invariants associated to  surgery theory and  metrics of positive scalar curvature, see \cite{XY1,XY2, WXY, Z}.

In \cite{DG1,DG2,DG3} Deeley and Goffeng produce a geometrical version of the analytic surgery exact sequence in the spirit of Baum’s geometric K-homology theory.

\bigskip

A further way of implementing the index map was given in \cite{connes-book} by Connes, where he used the so-called \emph{tangent groupoid}, also called by now \emph{adiabatic groupoid}.
In \cite{Zadiabatic}, the author of the present paper used  the group $K_*(C^*_r(G_{ad}^{[0,1)}))$ appearing in the exact sequence 
		\begin{equation}\label{5}
		\xymatrix{\dots\ar[r] & K_*(C^*_r(G\times(0,1)))\ar[r] & K_*(C^*_r(G_{ad}^{[0,1)}))\ar[r]^{\mathrm{ev}_0} & K_*(C^*_r(\mathfrak{A}G))\ar[r]& \dots}
		\end{equation}
	as a receptacle for K-theoretic secondary invariants. Here
$G$ is a Lie groupoid, $\mathfrak{A}G$ is its Lie algebroid and $G_{ad}^{[0,1)}$ is its adiabatic deformation.
The  case of a smooth manifold $\widetilde{X}$ with a proper and free $\Gamma$ action is realized by the particular groupoid $\widetilde{X}\times_\Gamma\widetilde{X}$. But this approach can also be applied to
other geometric situation encoded by a general Lie groupoid, such as foliations.

The main result of this paper is the following.

\begin{theorem}\label{0.1}
	Let $\widetilde{X}$ be a smooth Riemannian manifold equipped with a proper, free, isometric and cocompact action of a discrete group $\Gamma$. Let $G$ be the Lie groupoid $\widetilde{X}\times_\Gamma\widetilde{X}$.
	Then there exists a commutative diagram 
	\begin{equation}
	\xymatrix{\cdots\ar[r]&K_*(C_r^*(G)\otimes C_0(0,1))\ar[r]\ar[d]& K_*(C^*_r(G_{ad}^{[0,1)}))\ar[r]\ar[d] & K_*(C^*_r(TX))\ar[r]\ar[d]&\dots\\
		\cdots\ar[r]&K_{*+1}(C^*(\widetilde{X})^\Gamma)\ar[r] & K_{*+1}(D^*(\widetilde{X})^\Gamma)\ar[r]& K_{*+1}(D^*(\widetilde{X})^\Gamma/C^*(\widetilde{X})^\Gamma)\ar[r]&\cdots}
	\end{equation}
	such that the vertical arrows are isomorphisms.
\end{theorem}

 Using this result, we will see that the $\varrho$-classes  associated to a metric with positive scalar curvature defined in \cite{PS1} and \cite{Zadiabatic} correspond to each other through the middle vertical isomorphism.

In \cite{PZ} the methods from \cite{Zadiabatic} are used to define secondary invariants associated to metrics with positive scalar curvature on stratified manifolds and other singular situations such as foliations which degenerate on the boundary.
 In order to deal with the singularities, a slightly different exact sequence of groupoid C*-algebras is used. The proof of Theorem \ref{0.1} can be easily adapted to the context of stratified manifolds and we obtain the following result.
\begin{theorem}\label{0.2}
	Let ${}^\mathrm{S}X$ be a Thom-Mather stratified space  with fundamental group $\Gamma$ and let ${}^\mathrm{S}\widetilde{X}$ be its universal covering with the associated $\Gamma$-equivariant stratification. Let the regular part of ${}^\mathrm{S}\widetilde{X}$ be equipped with an incomplete iterated edge metric, then there exists a commutative diagram
		\begin{equation}
		\xymatrix{\cdots\ar[r]&K_*(C_r^*(G_{|\mathring{X}}\otimes C_0(0,1))\ar[r]\ar[d]& K_*(C^*_r(G_{ad}^{[0,1)}))\ar[r]\ar[d] & K_*(C^*_r(T_{\varphi}^{NC}X))\ar[r]\ar[d]&\dots\\
			\cdots\ar[r]&K_{*+1}(C^*(\widetilde{X})^\Gamma)\ar[r] & K_{*+1}(D^*(\widetilde{X})^\Gamma)\ar[r]& K_{*+1}(D^*(\widetilde{X})^\Gamma/C^*(\widetilde{X})^\Gamma)\ar[r]&\cdots}
		\end{equation}
		such that the vertical arrows are isomorphisms.
\end{theorem}	
The definition of the groupoids in the first row will be recalled in Section \ref{stratified-spaces}.
But it is worth noticing that  in the first row we make use  of non-compact spaces equipped with complete metrics, whereas the second row is constructed from compact spaces equipped with non-complete metrics and the two rows are related by a conformal change of metrics. 

		\tableofcontents 	
\subsubsection*{Aknowledgements}
It is a pleasure to thank Moulay Benameur, Paolo Piazza, Thomas Schick and Georges Skandalis for many interesting discussions. 
\section{Roe's Algebras}

In this section we are going to recall the fundamental definitions and results about coarse geometry, coarse C*-algebras and coarse index theory. We will not enter into the details of the proofs, which one can easily find in the literature.
Let $X$ be any set, if $A \subset X \times X$ and $B \subset X \times X$, we will use the following notation:
\[
A^{-1} := \{(y, x)\, |\, (x, y)\in A\}\]
and
\[A\circ B := \{(x, z) \,|\, \exists\, y\in X\, : (x, y)\in A\, \mbox{and}\, (y, z)\in B\}.\]

\begin{definition} A coarse structure on $X$ is a collection of subsets of $X\times X$, called entourages,
	that have the following properties:
	\begin{itemize}
		\item For any entourages $A$ and $B$, the subsets $A^{-1}$, $A\circ B$, and $A\cup B$ are entourages;
		\item Every finite subset of $X \times X$ is an entourage;
		\item Any subset of an entourage is an entourage.
	\end{itemize}
	If $\{(x,x) | x \in X\}$ is an entourage, then the coarse structure is said to be unital.
\end{definition}

\begin{definition}
	Let $(X,d)$ be a metric space and let $S$ be any set. Two function $f_1,f_2\colon S\to X$ are said close if $\{d(f_1(s),f_2(s))\,:\,s\in S\}$ is a bounded set of $\RR$.
\end{definition}

\begin{definition}
	Let $(X,d)$ be a metric space. A subset $E\in X\times X$ is said to be controlled if the projection maps $\pi_1,\pi_2\colon X\times X\to X$ are close.	
\end{definition}
The controlled sets are the ones that are contained in a uniformly bounded neighbourhood of the diagonal.
The \emph{metric coarse structure} on $(X,d)$ is given by the collection of all controlled subset of $X\times X$.

%
%
%

Let $\widetilde{X}$ be a proper metric space equipped with a free and proper action of a countable discrete group $\Gamma$ of isometries of $\widetilde{X}$.

\begin{definition}
	Let $H$ be a Hilbert space equipped with a representation
	\[\rho\colon C_0(\widetilde{X})\to \BB(H)\]
	and a unitary representation
	\[U\colon \Gamma\to \BB(H)\]
	such that $U(\gamma)\rho(f)=\rho(\gamma^{-1}f)U(\gamma)$ for every $\gamma\in\Gamma$ and $f\in C_0(\widetilde{X})$.
	We will call such a triple $(H,U,\rho)$ a $\Gamma$-equivariant $C_0(\widetilde{X})$-module.	
\end{definition}

\begin{example}
	Let us set $H=L^2(\widetilde{X},\mu)$, where $\mu$ is a $\Gamma$-invariant Borel measure on $\widetilde{X}$. 
	Put
	\begin{itemize}
		\item $\rho\colon C_0(\widetilde{X})\to \BB(H)$ the representation given by multipliation operators;
		\item $U\colon\Gamma\to\BB(H)$ the representation given by traslation $U_\gamma\varphi(x)=\varphi(\gamma^{-1}x)$ for every $x\in X$.
	\end{itemize}	
	Then $(H,U,\rho)$ is a $\Gamma$-equivariant $C_0(\widetilde{X})$-module.
\end{example}

\begin{definition}
	
	Let $A$ be a C*-algebra and let $H$ be a Hilbert space. A 
	resentation $\rho\colon A\to\BB(H)$ is said to be ample if it extends to a representation
	$\widetilde{\rho}\colon \widetilde{A}\to\BB(H)$ of the unitalization of A which has the following properties:
	\begin{itemize}
		\item $\widetilde{\rho}$ is non-degenerate, meaning $\widetilde{\rho}(\widetilde{A})H$ is dense in $H$;
		\item  $\widetilde{\rho}(a)$ is compact for $a\in\widetilde{A}$ if and only if $a = 0$.
	\end{itemize}
	
	Moreover we will say that a
	representation $\rho\colon A\to\BB(H)$ is very ample if it is the countable direct sum of a fixed ample representation.
\end{definition}
If $H$ is equipped with a unitary representation $U$ of $\Gamma$, then we say that an operator $T\in\BB(H)$ is $\Gamma$-equivariant if $U_\gamma T U_{\gamma^{-1}}=T$ for all $\gamma\in\Gamma$.

\subsection{Controlled operators}

\begin{definition}
	Let $X$ and $Y$ be two proper metric spaces; let
	$\rho_X\colon C_0(X)\to \BB(H_X)$ and $\rho_Y\colon C_0(Y)\to \BB(H_Y)$ be two representations on separable Hilbert spaces.
	\begin{itemize}
		\item The support of an element $\xi\in H_X$ is the set $\mathrm{supp}(\xi)$ of all $x\in X$ such that for every open neighbourhood $U$ of $x$ there is a function $f\in C_0(U)$ with $\rho_X(f)\xi\neq 0$.
		\item The support of an operator $T\in\BB(H_X,H_Y)$ is the set 
		$\mathrm{supp}(T)$ of all $(y,x)\in Y\times X$ such that for all
		open neighbourhoods $U\ni y$ and $V\ni x$ there exist $f\in C_0(U)$ and $g\in C_0(Y)$ such that  $\rho_Y(f)T\rho_X(g)\neq 0$.
		\item An operator $T\in\BB(H_X,H_Y)$ is properly supported if the slices
		$\{y\in Y\,:\, (y,x)\in \mathrm{supp}(T) \}$ and  $\{x\in X\,:\, (y,x)\in \mathrm{supp}(T) \}$ are closed sets.
	\end{itemize}
\end{definition}

\begin{definition}
	Let $X$ be as in the previous definition. An operator $T\in\BB(H_X)$ is said to be controlled if its support is a controlled subset of $X\times X$.
\end{definition}
This means that an operator is controlled if it is supported in a uniformly bounded neighbourhood of the diagonal of $X\times X$. These operators are also said to have finite propagation.
\begin{proposition}
	The set of all controlled operators for $\rho_X\colon C_0(X)\to H_X$ is a unital *-algebra of $\BB(H_X)$.
\end{proposition}

\subsection{The C*-algebras  $C^*(\widetilde{X})^\Gamma$ and $D^*(\widetilde{X})^\Gamma$}

Let $(H_{\widetilde{X}},U,\rho)$ be an ample $\Gamma$-equivariant $C_0(\widetilde{X})$-module.
\begin{definition}
	We define the C*-algebra $C^*(\widetilde{X})^{\Gamma}$ as the closure in $\BB(H_{\widetilde{X}})$ of the *-algebra of all $\Gamma$-equivariant operators $T$ such tat
	\begin{itemize}
		\item $T$ has finite propagation, i.e. there is a $R>0$ such that $\rho(\varphi)T\rho(\psi)=0$ for all $\varphi,\psi\in C_0(\widetilde{X})$ with $d(\mathrm{supp}(\varphi),\mathrm{supp}(\psi))>R$;
		\item $T$ is locally compact, i.e. $T\rho(\varphi)$ and $\rho(\varphi)T$ are compact operators for all $\varphi\in C_0(\widetilde{X})$.
	\end{itemize}
\end{definition}

\begin{definition}
	The algebra  $D^*(\widetilde{X})^\Gamma$ is the  closure in $\BB(H_{\widetilde{X}})$ of the *-algebra of all $\Gamma$-equivariant operators $T$ such that
	\begin{itemize}
		\item $T$ has finite propagation;
		\item $T$ is pseudolocal, i.e. $[T,\rho(\varphi)]$ is a  compact operator for any $\varphi\in C_0(\widetilde{X})$.
	\end{itemize}
	Here $H_{\widetilde{X}}$ has the structure of a very ample $\Gamma$-equivariant $\widetilde{X}$-module.
\end{definition}

If $\Gamma$ is the trivial group, then we will suppress it from the notation and write $C^*(\widetilde{X})$ and $D^*(\widetilde{X})$.
\begin{remark} 
The C*-algebra $D^*(\widetilde{X})^\Gamma$ is a *-subalgebra of the multiplier algebra of $C^*(\widetilde{X})^\Gamma$.
\end{remark}
\begin{remark}
	
	The algebras $C^*(\widetilde{X})^{\Gamma}$ and $D^*(\widetilde{X})^\Gamma$  depend on the $C_0(\widetilde{X})$-module used to construct it, but one can prove  that their K-theory does not.
\end{remark}

\subsection{The Paschke duality}

Since $C^*(\widetilde{X})^\Gamma$ is a bilateral *-ideal in $D^*(\widetilde{X})^\Gamma$,
we can consider the quotient C*-algebra $D^*(\widetilde{X})^\Gamma/C^*(\widetilde{X})^\Gamma$.
By a truncation argument one can prove the following isomorphism of C*-algebras
\begin{equation}\label{isoD/C}
D^*(\widetilde{X})^\Gamma/C^*(\widetilde{X})^\Gamma\cong D^*(X)/C^*(X)
\end{equation}
where $X$ is the quotient space $\widetilde{X}/\Gamma$.

The Paschke duality asserts the existence of the following isomorphism
\begin{equation}
\mathcal{P}\colon K_{0}(D^*(X)/C^*(X))\to KK_1(C(X),\CC),
\end{equation}
that sends  the projection $p\in D^*(X)/C^*(X)$ to the Kasparov bimodule
$(H,\rho, 2p-1)$, where $\rho\colon C(X)\to H$ is the $C(X)$-module used to define $D^*(X)/C^*(X)$.

We can see $\mathcal{P}$ as an asymptotic morphism in $E^1(D^*(X)/C^*(X)\otimes C(X),\CC)$, see \cite{QR}.
Indeed consider the generator $u$ of $E^1(\mathbb{Q}(H), \CC)$, where $\mathbb{Q}(H)$ is the Calkin algebra of $H$.  It is  the class associated the following extension of C*-algebra
\[
\xymatrix{0\ar[r] & \mathbb{K}(H)\ar[r] & \mathbb{B}(H)\ar[r]& \mathbb{Q}(H)\ar[r]& 0}
\]
and $u$ is given by the boundary map of the long exact sequence in E-theory  associated to the previous exact sequence.

Let $\mu\colon D^*(X)/C^*(X)\otimes C(X)\to \mathbb{Q}(H)$ be the *-homomorphism given by $$T\otimes f\mapsto M_fT,$$ where $M_f$ is the multiplication operator. It is a well-defined *-homomorphism because  $D^*(X)/C^*(X)$ and $C(X)$ commute in $\mathbb{Q}(H)$.
Then $\mathcal{P}$ is given by 
\begin{equation}\label{paschke}
\mu^*(u)\in E^1(D^*(X)/C^*(X)\otimes C(X),\CC).
\end{equation}
More precisely $\mu^*(u)$  arises from  the pull-back extension
\[
\xymatrix{
	0\ar[r] & \mathbb{K}(H)\ar[r]\ar[d]^= & E\ar[r]\ar[d]& D^*(X)/C^*(X)\otimes C(X)\ar[d]^\mu\ar[r]& 0\\
	0\ar[r] & \mathbb{K}(H)\ar[r] & \mathbb{B}(H)\ar[r]& \mathbb{Q}(H)\ar[r]& 0}
\]
 as the boundary morphism of the long exact sequence in E-theory associated to the top row.
 \subsection{The analytic surgery exact sequence}
 
 Let $\widetilde{X}$ be a proper metric space such that the countable discrete group $\Gamma$ acts properly, freely and isometrically on it.
 The algebras defined in the previous section fit in the following exact sequence

 \begin{equation}\label{ASES}
 \xymatrix{\cdots\ar[r] & K_*(C^*(\widetilde{X})^\Gamma)\ar[r]&  K_*(D^*(\widetilde{X})^\Gamma)\ar[r] &  K_*(D^*(\widetilde{X})^\Gamma /C^*(\widetilde{X})^\Gamma)\ar[r] & \cdots}
 \end{equation}
 the Higson-Roe analytic surgery exact sequence.
Notice that $ K_*(C^*(\widetilde{X})^\Gamma)$ is isomorphic to 
$ K_*(C^*_r(\Gamma))$ and recall that $K_*(D^*(\widetilde{X})^\Gamma /C^*(\widetilde{X})^\Gamma)$ is isomorphic to  $KK_{*-1}(C_0(X), \CC)$ by Paschke duality.
In \cite{Roeassembly}  Roe proves that the boundary morphism of \eqref{ASES} is equivalent to the assembly map. In other words the following diagram
\begin{equation}
\xymatrix{
	K_{*+1}(D^*(\widetilde{X})^\Gamma /C^*(\widetilde{X})^\Gamma)\ar[r]^(.6)\partial\ar[d]^{\mathcal{P}}& K_*(C^*(\widetilde{X})^\Gamma)\ar[d]^\cong\\
	K_*^\Gamma(\widetilde{X})\ar[r]^{\mu_X^\Gamma}& K_*(C^*_r(\Gamma))
	}
\end{equation}
is commutative. Here we used the fact that, because of the assumptions about the action of $\Gamma$ on $\widetilde{X}$, the equivariant K-homology group  $KK_*^\Gamma(C_0(\widetilde{X}),\CC)$ is isomorphic to
 $KK_*(C_0(X),\CC)$.
 
\section{The Adiabatic Groupoid and the Gauge Adiabatic Groupoid}

We refer the reader to \cite{DL} and bibliography inside it for the notations and a detailed overview about groupoids and index theory.

\subsection{Lie groupoids and Lie algebroids}
\begin{definition} Let $G$ and $G^{(0)}$ be two sets.  A groupoid structure on $G$ over $G^{(0)}$ is given by the following morphisms:
	\begin{itemize}
		\item An injective map $u:G^{(0)}\rightarrow G$ called the unit map. We can identify $G^{(0)}$ with its image
		in $G$. 
		\item Two surjective maps: $r,s: G\rightarrow G^{(0)}$,
		which are respectively the range and  source map.
		\item An involution: $ i: G\rightarrow G
		$, $  \gamma  \mapsto \gamma^{-1} $ called the inverse
		map. It satisfies: $s\circ i=r$.
		\item A map $ p: G^{(2)}  \rightarrow  G
		$, $ (\gamma_1,\gamma_2)  \mapsto  \gamma_1\cdot \gamma_2 $
		called the product, where the set 
		$$G^{(2)}:=\{(\gamma_1,\gamma_2)\in G\times G \ \vert \
		s(\gamma_1)=r(\gamma_2)\}$$ is the set of composable pair. Moreover for $(\gamma_1,\gamma_2)\in
		G^{(2)}$ we have $r(\gamma_1\cdot \gamma_2)=r(\gamma_1)$ and $s(\gamma_1\cdot \gamma_2)=s(\gamma_2)$.
	\end{itemize}
	
	The following properties must be fulfilled:
	\begin{itemize}
		\item The product is associative: for any $\gamma_1,\
		\gamma_2,\ \gamma_3$ in $G$ such that $s(\gamma_1)=r(\gamma_2)$ and
		$s(\gamma_2)=r(\gamma_3)$ the following equality
		holds $$(\gamma_1\cdot \gamma_2)\cdot \gamma_3= \gamma_1\cdot
		(\gamma_2\cdot \gamma_3)\ .$$
		\item For any $\gamma$ in $G$: $r(\gamma)\cdot
		\gamma=\gamma\cdot s(\gamma)=\gamma$ and $\gamma\cdot
		\gamma^{-1}=r(\gamma)$.
	\end{itemize}
	
	We denote a groupoid structure on $G$ over $G^{(0)}$ by
	$G\rightrightarrows G^{(0)}$,  where the arrows stand for the source
	and target maps. 
\end{definition}

We will adopt the following notations: $$G_A:=
s^{-1}(A)\ ,\ G^B=r^{-1}(B) \ \mbox{ and } G_A^B=G_A\cap G^B \,$$
in particular if $x\in G^{(0)}$, the  $s$-fiber (resp. 
$r$-fiber) of $G$ over $x$ is $G_x=s^{-1}(x)$ (resp. $G^x=r^{-1}(x)$).

\begin{definition}
	
	We call $G$ a Lie groupoid when $G$ and $G^{(0)}$ are second-countable smooth manifolds
	with $G^{(0)}$ Hausdorff, the structural homomorphisms are smooth, $u$ is an embedding,
	$r$ and $s$ are submersions, and $i$ is a diffeomorphism.
\end{definition}

	


\begin{definition} A  Lie
	algebroid $\mathfrak{A} =(p:\mathfrak{A}\rightarrow TM,[\ ,\ ]_{\mathfrak{A}})$ on a smooth
	manifold $M$ is a vector bundle $\mathfrak{A} \rightarrow M$
	equipped with a bracket $[\ ,\ ]_{\mathfrak{A}}:\Gamma(\mathfrak{A})\times \Gamma(\mathfrak{A})
	\rightarrow \Gamma(\mathfrak{A})$ on the module of sections of $\mathfrak{A}$, together
	with a homomorphism of fiber bundle $p:\mathfrak{A} \rightarrow TM$ from $\mathfrak{A}$ to the
	tangent bundle $TM$ of $M$, called the  anchor map, fulfilling the following conditions:
	\begin{itemize}
		\item the bracket $[\ ,\ ]_{\mathfrak{A}}$ is $\RR$-bilinear, antisymmetric
		and satisfies the Jacobi identity,
		\item $[X,fY]_{\mathfrak{A}}=f[X,Y]_{\mathfrak{A}}+p(X)(f)Y$ for all $X,\ Y \in
		\Gamma(\mathfrak{A})$ and $f$ a smooth function of $M$, 
		\item $p([X,Y]_{\mathfrak{A}})=[p(X),p(Y)]$ for all
		$X,\ Y \in \Gamma(\mathfrak{A})$.
	\end{itemize}
\end{definition}

The  tangent space to $s$-fibers, that is $T_sG := \ker ds =
\bigcup_{x\in G^{(0)}} TG_x$ has the structure of  a Lie algebroid
on $G^{(0)}$, with anchor map given by $dr$. 
It is denoted by 
$\mathfrak{A}G$ and we call it the Lie algebroid of $G$.
One can prove that it is isomorphic to the normal bundle of the inclusion $G^{(0)}\hookrightarrow G$. 
\subsection{The adiabatic groupoid and the gauge adiabatic groupoid}\label{DNC}
Let $M_0$ be a smooth compact submanifold of a smooth manifold $M$ with normal bundle $\mathcal{N}$.
As a set, the deformation to the normal cone is \begin{equation}
DNC(M_0,M):=\mathcal{N} \times\{0\}\sqcup M\times\RR .
\end{equation}
In order to recall its smooth structure, we fix an exponential map, which is a diffeomorphism $\theta$ from
a neighbourhood $V'$ of the zero section $M_0$ in $N$ to a neighbourhood $V$ of $M_0$ in $M$.
We may cover $DNC(M_0,M)$ with two open sets $M\times \RR^*$, with the product differentiable structure, and $W=\mathcal{N}\times {0}\sqcup V\times\RR^*$,
endowed with the differentiable structure for which the map 
\begin{equation}\label{adiabatictopology}
\Psi\{(m,\xi,t)\in \mathcal{N}\times\RR\,|\,(m,t\xi)\in V'\}\to W
\end{equation}
given by $(m,\xi,t)\mapsto(\theta(m, t\xi),t)$,
for $t\neq0$, and by $(m,\xi,0)\mapsto(m,\xi,0)$, for $t=0$, is a diffeomorphism. One can verify that the transition map on the overlap of thes two charts is smooth, see for instance \cite[Section 3.1]{HS}.

\begin{definition}
	
The adiabatic groupoid $G^{[0,1]}_{ad}$ is given by the groupoid 
\[
\mathfrak{A}G\times\{0\}\cup G\times(0,1]\rightrightarrows G^{(0)}\times [0,1],
\]
with the smooth structure given by the 
deformation to the normal cone associated to the embedding  $G^{(0)}\hookrightarrow G$.
	We will use the notation $G_{ad}^{[0,1)}$ for the restriction of the adiabatic groupoid to the interval open at $1$, given by
	\[
	\mathfrak{A}G\times\{0\}\cup G\times(0,1)\rightrightarrows G^{(0)}\times [0,1).
	\]
\end{definition}

\begin{remark}\label{KKeq}
Let $\mathrm{ev}_0\colon C_r^*(G^{[0,1]}_{ad})\to C^*_r(\mathfrak{A}G)$ be the evaluation at $t=0$, then the associated KK-element is a KK-equivalence.
Indeed notice that $C^*_r(\mathfrak{A}G)$ is nuclear and that the kernel of $\mathrm{ev}_0$ is isomorphic to the contractible C*-algebra $C^*_r(G)\otimes C_0(]0,1])$.
Then $[\mathrm{ev}_0]\colon KK(A, C_r^*(G^{[0,1]}_{ad}))\to KK\left(A, C^*_r(\mathfrak{A}G)\right)$, understood as the Kasparov product with $[\mathrm{ev}_0]$ on the right, is an isomorphism for all C*-algebras $A$. This implies that there exists an element $[\mathrm{ev}_0]^{-1}\in KK\left(C^*_r(\mathfrak{A}G),C_r^*(G^{[0,1]}_{ad})\right)$ such that $[\mathrm{ev}_0]^{-1}\otimes [\mathrm{ev}_0]=1_{C^*_r(\mathfrak{A}G)}$ and $[\mathrm{ev}_0]\otimes [\mathrm{ev}_0]^{-1}=1_{C_r^*(G^{[0,1]}_{ad})}$. 
\end{remark}

Now we recall a definition from \cite[section 2.3]{DS}.
We have a natural action of the group $R_+^*$ compatible with the groupoid structure on $G_{ad}^{[0,1)}$ defined as follows.
Let $\alpha\colon \RR_+^*\to (0,1)$ be defined as $\alpha(t,\lambda)=\frac{2}{\pi}\arctan\left(\lambda\tan\left(\frac{\pi}{2}t\right)\right)$,  then one can easily check that $\alpha(\alpha(t,\lambda),\lambda')=\alpha(t,\lambda\lambda')$.

Thus we have that the map defined by
$(\gamma,t ; \lambda)\mapsto (\gamma,\alpha(t,\lambda))$ for $t\neq0$ and $(x,V,0;\lambda)\mapsto(x,\frac{1}{\lambda}V,0)$ gives an action of $\RR_+^*$ on  $G_{ad}^{[0,1)}$.
Notice that this action is isomorphic to the action on $G_{ad}$ from \cite[section 2.3]{DS}.
\begin{definition}
The gauge adiabatic groupoid $G_{ga}^{[0,1)}$ is the Lie groupoid obtained as the crossed product of this action:
\[
G_{ga}^{[0,1)}:=G_{ad}^{[0,1)}\rtimes\RR^*_+\rightrightarrows G^{(0)}\times[0,1).
\]
\end{definition}

\subsection{Groupoid C*-algebras}
We can associate to a Lie groupoid $G$ the *-algebra $C^\infty_c(G,\Omega^{\frac{1}{2}}(\ker ds\oplus\ker dr))$ of the compactly supported sections of the half densities bundle associated to $\ker ds\oplus\ker dr$, with:

\begin{itemize}
	\item the involution given by $f^*(\gamma)=\overline{f(\gamma^{-1})}$;
	\item and the product defined as $f*g(\gamma)=\int_{G_{s(\gamma)}} f(\gamma\eta^{-1})g(\eta)d\eta$.
\end{itemize}  

For all $x\in G^{(0)}$ the algebra $C^\infty_c(G,\Omega^{\frac{1}{2}}(\ker ds\oplus\ker dr))$ can be represented on 
$L^2(G_x,\Omega^{\frac{1}{2}}(G_x))$ by 
\[\lambda_x(f)\xi(\gamma)=\int_{G_{x}} f(\gamma\eta^{-1})g(\eta)d\eta, \]
where $f\in C^\infty_c(G,\Omega^{\frac{1}{2}}(\ker ds\oplus\ker dr))$ and $\xi\in L^2(G_x,\Omega^{\frac{1}{2}}(G_x))$.

\begin{definition}
	The reduced C*-algebra of a Lie groupoid G, denoted by $C^*_r(G)$, is the completion of $C^\infty_c(G,\Omega^{\frac{1}{2}}(\ker ds\oplus\ker dr))$ with respect to the norm
	\[
	||f||_r=\sup_{x\in G^{(0)}}||\lambda_x(f)||.
	\]
	
	The full C*-algebra of $G$ is the completion of 
	$C^\infty_c(G,\Omega^{\frac{1}{2}}(\ker ds\oplus\ker dr))$ with respect to all continuous representations.
\end{definition}

\begin{definition}
	Let $G$ be a Lie groupoid, then we associate to it a short exact sequence of C*-algebras
	\begin{equation}\label{AES}
	\xymatrix{0\ar[r] & C^*_r(G\times(0,1))\ar[r] & C^*_r(G_{ad}^{[0,1)})\ar[r]^{\mathrm{ev}_0} & C^*_r(\mathfrak{A}G)\ar[r]& 0}
	\end{equation}
	and we are going to call the following long exact sequence in K-theory
		\begin{equation}\label{KAES}
		\xymatrix{\dots\ar[r] & K_*(C^*_r(G\times(0,1)))\ar[r] & K_*(C^*_r(G_{ad}^{[0,1)}))\ar[r]^{\mathrm{ev}_0} & K_*(C^*_r(\mathfrak{A}G))\ar[r]& \dots}
		\end{equation}
 the (reduced) adiabatic exact sequence of $G$.
\end{definition}

The boundary map of  \ref{KAES} is given by the composition of the KK-element 
\begin{equation}\label{indad}
[\mathrm{ev}_0]^{-1}\otimes[\mathrm{ev}_1]\in KK\left(C_r^*(\mathfrak{A}G),C_r^*(G)\right).
\end{equation}
and  the suspension isomorphism $S$. Finally, notice that there is an analogous extension given by the full groupoid C*-algebras.

It is worth to point out that the Thom-Connes isomorphism \cite{Connes-thom, FS} gives a natural isomorphism of long exact sequences of KK-groups
\begin{equation}\label{TC}
\xymatrix@=0.8em{\dots\ar[r]& KK^*(A,C^*_r(G)\otimes C_0(0,1))\ar[r]\ar[d]^{\mathrm{TC}}& KK^*(A,C^*_r(G_{ad}^{[0,1)}))\ar[r]\ar[d]^{\mathrm{TC}}&KK^*(A,C^*_r(\mathfrak{A}G))\ar[r]\ar[d]^{\mathrm{TC}}&\dots\\
\dots\ar[r]& KK^{*+1}(A,C^*_r(G)\otimes\KK)\ar[r]& KK^{*+1}(A,C^*_r(G_{ga}^{[0,1)}))\ar[r]&KK^{*+1}(A,C^*_r(\mathfrak{A}G)\rtimes\RR^*_+)\ar[r]&\dots\\
}
\end{equation}
for any separable C*-algebra $A$, where the vertical arrows are given by the Kasparov product by the element constructed in \cite[Proposition1(i)]{FS}.
Notice that $$C^*_r(G)\otimes C_0(0,1)\rtimes\RR_+^*\cong C^*_r(G)\otimes \KK.$$
\subsection{The  Lie groupoid $\widetilde{X}\times_\Gamma\widetilde{X}$}
 \label{fundamentalgroupoid} 
 
Let $\pi\colon \widetilde{X}\to X$ be a Galois $\Gamma$-covering. 
Then the diagonal action of $\Gamma$ on $\widetilde{X}\times\widetilde{X}$ is proper and free. Let $G=\widetilde{X}\times_\Gamma\widetilde{X}$ be the quotient of this action, it has a Lie groupoid structure over $X$ described as follows:
\begin{itemize}
	\item the source and the range are given by $s([\widetilde{x}_1,\widetilde{x}_2])=\pi(\widetilde{x}_2)$ and $r([\widetilde{x}_1,\widetilde{x}_2])=\pi(\widetilde{x}_1)$
	\item the product of $[\widetilde{x}_1,\widetilde{x}_2]$ and $[\widetilde{x}_3,\widetilde{x}_4]$ is given by $[\widetilde{x}_1,\gamma(\widetilde{x}_2,\widetilde{x}_3)\cdot\widetilde{x}_4]$, where $\gamma(\widetilde{x}_2,\widetilde{x}_3)$ is the element of $\Gamma$ that sends $\widetilde{x}_3$ to $\widetilde{x}_2$.
\end{itemize}

The reduced C*-algebra $C^*_r(\widetilde{X}\times\widetilde{X})$ is the  C*-closure of the $C_c^\infty(\widetilde{X}\times\widetilde{X})$. One can see that simply as $\Gamma$-equivariant smoothing kernels on the universal covering of $X$.
It is easy to prove that $C^*_r(\widetilde{X}\times\widetilde{X})$ is Morita equivalent to $C^*_r(\Gamma)$.

The Lie algebroid of $G$ is isomorphic to $TX$, the tangent bundle of $X$, and the anchor map is given by the identity.
The reduced C*-algebra of the tangent bundle $C^*_r(TX)$ is isomorphic to the C*-algebra $C_0(T^*X)$ of continuous function vanishing at infinity. This isomorphism is given by the fiber-wise Fourier transform. 

Let us denote by $\partial_X^\Gamma\in KK(C_0(T^*X), C^*_r(\Gamma))$ the element defined in \eqref{indad} (up to Fourier transform and Morita equivalences).
In \cite{MP} it is proved that the application induced by the Kasparov product with $\partial^\Gamma_X$ is the $\Gamma$-equivariant analytical index of Atiyah-Singer.

\subsection{Poincaré duality}\label{pncr}

Let us consider the Lie groupoid of the pairs $X\times X$ over $X$. The groupoid C*-algebra $C^*_r(X\times X)$ is isomorphic to $\KK(L^2(X))$, the algebra of compact operators on $L^2(X)$. Its Lie algebroid is still $TX$ and 
\begin{equation}\label{dirac}
\cdot\otimes\partial_X\colon KK^*(\CC,C_0(T^*X))\to KK^*(\CC,\CC)
\end{equation}
is equivalent to the analytic index of Atiyah-Singer. 

Now, let $m\colon C_0(T^*X)\otimes C(X)\to C_0(T^*X)$ be the morphism given by 
$$\sigma\otimes f\mapsto \sigma\cdot \pi^*f,$$
where $\pi$ is the bundle projection.
  Then the so-called \emph{Dirac element}
$$D_X:=[m]\otimes_{C_0(T^*X)} \partial_X\in KK(C_0(T^*X)\otimes C(X),\CC)$$ implements, by Kasparov product, the Poincaré duality 
\begin{equation}\label{pduality}
\cdot\otimes D_X\colon KK(\CC,C_0(T^*X))\to KK(C(X),\CC)
\end{equation}
whose inverse is given by the principal symbol map.
See \cite{Kasparov, CS, DLisolated} for a detailed proof.

\subsection{Pseudodifferential operators}\label{PDO}

Let us recall the definition of a pseudodifferential $G$-operator.
We refer the reader to \cite{NWX} and \cite{Vassout} for pseudodifferential calculus on Lie groupoids.

\begin{definition}
	A linear $G$-operator is a continuous linear map
	$$ P\colon C^\infty_c(G,\Omega^{\frac{1}{2}})\to C^\infty(G,\Omega^{\frac{1}{2}})$$
	such that:
	\begin{itemize}
		\item $P$ restricts to a continuous family $(P_x)_{x\in G^{(0)}}$ of linear operators $P_x\colon C^\infty_c(G_x,\Omega^{\frac{1}{2}})\to C^\infty(G_x,\Omega^{\frac{1}{2}})$
		such that
		$$Pf(\gamma)= P_{s(\gamma)}f_{s(\gamma)}(\gamma) \quad \forall f\in C^\infty_c(G,\Omega^{\frac{1}{2}})$$
		where $f_x$ denotes the restriction of $f$ to $G_x$.
		\item The following equivariance property holds:
		$$ U_{\gamma}P_{s(\gamma)} =P_{r(\gamma)}U_{\gamma},  $$
		where $U_\gamma$ is the map induced on functions by the right multiplication by $\gamma$.
	\end{itemize}
	
	A linear G-operator $P$ is pseudodifferential of order $m$ if 
	\begin{itemize}
		\item  its  Schwartz kernel $k_P$ is  smooth outside $G^{(0)}$;
		\item for every distinguished chart $\psi: U\subset G\to \Omega\times s(U)\subset\RR^{n-p}\times\RR^{p}$ of $G$:
		$$  \xymatrix{U\ar[rr]^{\psi}\ar[dr]_{s} & &\Omega\times s(U)\ar[dl]^{p_{2}} \\
			&s(U)&  }$$ 
		the operator $(\psi^{-1})^{*}P\psi^{*}:
		C^{\infty}_{c}(\Omega\times s(U))\to C^{\infty}_{c}(\Omega\times
		s(U))$ is a smooth family parametrized by $s(U)$ of  
		pseudodifferential operators of order $m$ on $\Omega$.     
	\end{itemize} 	
	
	We say that $P$ is smoothing if $k_P$ is smooth and that $P$ is compactly supported if $k_P$ is compactly supported on $G$.
\end{definition}

The space $\Psi^*_c(G)$ of the compactly supported pseudodifferential $G$-operators is an involutive algebra. Observe that a pseudodifferential $G$-operator induces a family of pseudodifferential operators on $s$-fibers.
So we can define the principal symbol of a pseudodifferential $G$-operator $P$ as a function $\sigma(P)$ on $\mathfrak{S}^*G$, the cosphere bundle associated to the Lie algebroid $\mathfrak{A}G$ by
$$\sigma(P)(x,\xi)=\sigma(P_x)(x,\xi),$$
where $\sigma_{pr}(P_{x})$ is the principal symbol of the pseudodifferential
operator $P_{x}$ on the manifold $G_{x}$.  
Conversely,  given a symbol $f$ of order $m$ on $\mathfrak{A}^{*}G$ together with the
following data:
\begin{enumerate}
	\item a smooth embedding  $\theta:  U\to \mathfrak{A}G$,  where $ U$ is an open
	set in $G$ containing $G^{(0)}$, such that
	$\theta(G^{(0)})=G^{(0)}$, $(d\theta)|_{G^{(0)}}=\hbox{Id}$ and $\theta(\gamma)\in\mathfrak{A}_{s(\gamma)}G$
	for all $\gamma\in U$;
	\item a smooth compactly supported map $\phi:G\to \RR_{+}$ such that
	$\phi^{-1}(1)=G^{(0)}$;
\end{enumerate}
then a $G$-pseudodifferential operator $P_{f, \theta, \phi}$ is obtained by the
formula:
$$ P_{f, \theta, \phi}u(\gamma)=
\int_{\gamma'\in G_{s(\gamma)}\,,\, \xi\in \mathfrak{A}^{*}_{r(\gamma)}G}
e^{-i\theta(\gamma'\gamma^{-1})\cdot \xi}f(r(\gamma),
\xi)\phi(\gamma'\gamma^{-1})u(\gamma'))$$
with $u\in C^{\infty}_{c}(G,\Omega^{\frac{1}{2}})$. 
The principal symbol of $P_{f, \theta, \phi}$ is just the leading
part of $f$.  

The principal symbol map respects pointwise product while the product
law for total symbols is much more involved. An operator is {\it
	elliptic} when its principal symbol never vanishes and in that
case,  as in the classical situation, 
it has a parametrix inverting it modulo
$\Psi^{-\infty}_{c}(G)=C^{\infty}_{c}(G)$. 
\begin{remark}
	All these definitions and properties immediately extend to the case of operators acting
	between sections of bundles on $G^{(0)}$ pulled back to $G$ with the
	range map $r$. The space of compactly supported pseudodifferential
	operators on  $G$ acting on sections of $r^*E$ and taking values in sections of $r^*F$
	will be noted $\Psi_c^*(G,E, F)$. If $F=E$ we get an algebra
	denoted by $\Psi_c^*(G,E)$. 
\end{remark}

The operators of zero order $\Psi_c^0(G)$ form a subalgebra of the multiplier algebra
$M(C^{*}_r(G))$ and we will denote by $\Psi^0(G)$ its closure in the C*-norm. Moreover the closure of the operators of negative order is
$C^{*}_r(G)$.

From now on $G$ will be the Lie groupoid $\widetilde{X}\times_\Gamma\widetilde{X}\rightrightarrows X$, where $\Gamma $ acts on $\widetilde{X}$ freely and properly with $X=\widetilde{X}/\Gamma$.

\begin{remark}
	In our particular case it turns out  that the algebra of 0-order pseudodifferential $G$-operators is nothing but the algebra $\Psi^0_{\Gamma,prop}(\widetilde{X})$ of properly supported $\Gamma$-invariant pseudodifferential operator on $\widetilde{X}$, see \cite[Example 4.4]{NWX}. We will denote by $\Psi^0_{\Gamma}(\widetilde{X})$ its C*-closure.
\end{remark}

As in the classical case, one has the following pseudodifferential extension
\[
\xymatrix{0\ar[r] & C^*_r(G)\ar[r] & \Psi^0(G)\ar[r]^(.45)\sigma & C(\mathfrak{S}^*G)\ar[r]& 0}
\]
where the role of compact operators is played by the groupoid C*-algebra and symbols are functions on the cosphere bundle of the Lie algebroid.

If we take the pseudodifferential extension of the adiabatic groupoid we have the following short exact sequence.
\begin{equation}\label{pdoad}
\xymatrix{0\ar[r] & C^*_r(G_{ad}^{[0,1)})\ar[r] & \Psi^0(G_{ad}^{[0,1)})\ar[r]^(.4)\sigma & C(\mathfrak{S}^*(G_{ad}^{[0,1)}))\ar[r]& 0}
\end{equation}
Since $\mathfrak{A}(G_{ad}^{[0,1)})$ is isomorphic to $\mathfrak{A}(G)\times [0,1)$, it follows that $K_*(C(\mathfrak{S}^*(G_{ad}^{[0,1)})))$ is trivial and then, by exactness, one has that the first arrows of \eqref{pdoad}  induces the isomorphism
\begin{equation}\label{isoad}
K_*(C^*_r(G_{ad}^{[0,1)}))\cong K_*(\Psi^0(G_{ad}^{[0,1)})).
\end{equation}

Let us investigate then more closely the algebra $\Psi^0(G_{ad}^{[0,1)})$. It is a $C_0([0,1))$-algebra such that
\begin{itemize}
	
	\item at $t\neq0$ we have the algebra $\Psi^0(G)$, that is isomorphic to $\Psi^0_{\Gamma}(\widetilde{X})$; so for $t\in (0,1)$ we have a path $P_t$ of $\Gamma$-equivariant operators on $\widetilde{X}$ such that $P_1=0$ and the propagation of $P_t$ goes to $0$ as $t$ goes to $0$ (recall the differential structure of the adiabatic deformation given by \eqref{adiabatictopology});
	
	\item at $t=0$ we have $\Psi^0(TX)$, where we are seeing $TX\rightrightarrows X$ as a Lie groupoid. Since the source and the target map are the same for $TX$, it turns out that an element in $\Psi^0(TX)$ is a family of $\RR^k$-invariant pseudodifferential operators on $\RR^{k}$, with $k=\dim X$.
	Since a pseudodifferential operator is uniquely determined by its total symbol 
	and since we are considering polyhomogeneous symbols, it is easy to check that $\Psi^0_{\RR^k}(\RR^k)$ is isomorphic to the closure of $S^0(\RR^k)^{\RR^k}$, the algebra of the $\RR^k$-equivariant symbols on $\RR^k$.
But this algebra is isomorphic to the algebra of continuous functions on the closed unit ball $B^k$.
	Hence at $t=0$ we have the algebra $C(\mathfrak{B}^*X)$ of the continuous functions on the co-disk bundle of $X$.	
\end{itemize}

	Consider the map $\mathfrak{m}\colon C(X)\to \Psi^0_{\Gamma}(\widetilde{X})$ which associates to a function $f$ the operator $\mathfrak{m}(f)$ of multiplication by $f$. 
	The mapping cone C*-algebra of $\mathfrak{m}$ is given by 
	\[
	\mathcal{C}_{\mathfrak{m}}:=\{(f,P_t)\in C(X)\oplus\Psi^0_{\Gamma}(\widetilde{X})[0,1)\,|\, P_0=\mathfrak{m}(f) \}.
	\]
    
   Observe that such a path $P_t$ defines an element in $ \Psi^0(G_{ad}^{[0,1)})$, then we have the following  *-homomorphism $\eta\colon\mathcal{C}_{\mathfrak{m}}\to\Psi^0(G_{ad}^{[0,1)})$.
   
\begin{lemma}\label{isopdo}
The *-homomorphism
	$\eta$ induces an isomorphism
	\[
	[\eta]\colon K_*(\mathcal{C}_{\mathfrak{m}}) \to K_*(\Psi^0(G_{ad}^{[0,1)})).
	\]
\end{lemma}

\begin{proof}
	The following commutative diagram
	\[\xymatrix{0\ar[r]& \Psi^0_{\Gamma}(\widetilde{X})\otimes C_0(0,1)\ar[r]\ar[d]&\mathcal{C}_{\mathfrak{m}}\ar[r]^{\mathrm{ev}_0}\ar[d]^{\eta}& C(X)\ar[r]\ar[d]&0\\
		0\ar[r] &\Psi^0_\Gamma(\widetilde{X})\otimes C_0(0,1)\ar[r]& \Psi^0(G_{ad}^{[0,1)})\ar[r]^{\mathrm{ev}_0}&\Psi^0(TX)\ar[r]&0}\]
has exact rows. Moreover,  up to the isomorphism between $\Psi^0(TX)$ and $C(\mathcal{B}^*X)$,  the right vertical arrow is exactly given by the pull-back of functions induced by $\pi\colon \mathcal{B}^*X\to X$.
Since $\pi$ is a homotopy equivalence, $\pi^*$ induces an isomorphism in K-theory. By the Five Lemma, it follows that $\eta$ induces an isomorphism of $K$-groups.
\end{proof}
  
\section{The main theorem}

From now on let $G\rightrightarrows X$ be the Lie groupoid $\widetilde{X}\times_\Gamma\widetilde{X}$ of subsection \ref{fundamentalgroupoid}. 
In this section we are going to compare the 	
adiabatic exact sequence \eqref{KAES} associated to $G$ and the the analytic surgery exact sequence \eqref{ASES} for $\widetilde{X}$ and we establish an explicit isomorphism between them.

\subsection{First approach}\label{first}
First consider the Hilbert space $\mathcal{H}:=L^2(\widetilde{X}\times(0,1))$. It is an ample $\Gamma$-equivariant $C_0(\widetilde{X})$-module.
Now observe that the essential *-ideal $C^*_r(G)\otimes C_0(0,1)\rtimes \RR_+^*$ of $C^*_r(G_{ga}^{[0,1)})$ is isomorphic to the subalgebra  $C^*(\widetilde{X})^\Gamma$ of $\mathbb{B}(\mathcal{H})$. This implies that $C^*_r(G_{ga}^{[0,1)})$ is faithfully represented on $\mathcal{H}$ through a *-homomorphism
\[
\iota\colon C^*_r(G_{ga}^{[0,1)})\to \BB(\mathcal{H}).
\]	

\begin{remark}\label{rmk-iota}
One can see the C*-algebra $C^*_r(G)\otimes C_0(0,1)\rtimes \RR_+^*\cong C^*_r(G)\otimes \KK(L^2(0,1))$  as the $\Gamma$-equivariant elements of a subalgebra sitting inside the multipliers of the groupoid C*-algebra of $\widetilde{G}:=\widetilde{X}\times\widetilde{X}\times (0,1)\times (0,1)\rightrightarrows\widetilde{X}$.
Notice that, although one is tempted to say that $C^*_r(G)\otimes \KK(L^2(0,1))$ is the $\Gamma$-equivariant part of $C^*_r(\widetilde{G})$ it-self, this is not true: indeed $C^*_r(\widetilde{G})$ is defined as the closure of compactly supported elements and it is isomorphic to $\KK(L^2(\mathcal{H}))$, whereas the equivariant lifts of elements in   $C^*_r(G)\otimes \KK(L^2(0,1))$  are supported near the diagonal: in other words they are properly supported, but not compactly supported.
The same reasoning holds for $C^*_r(G_{ga}^{[0,1)})$, which is the $\Gamma$-equivariant part of a subalgebra in the multipliers of $C^*_r(\widetilde{G}_{ga}^{[0,1)})$.

Finally, observe that
if $\widetilde{\xi}\in C^\infty(\widetilde{G}_{ga}^{[0,1)})$ is the $\Gamma$-equivariant lift of an element $\xi\in C^\infty_c(G_{ga}^{[0,1)})$, then $\iota(\xi)$ is the image of the lift $\widetilde{\xi}$ through the extension to multiplier algebras of the morphism $\widetilde{\iota}\colon C^*_r(\widetilde{G})\otimes C_0(0,1)\rtimes \RR_+^*\to \KK(L^2(\mathcal{H}))$.
\end{remark}

\begin{lemma}\label{mappa}
	The image of $\iota$ is contained in  $D^*(\widetilde{X})^\Gamma$.
\end{lemma}
\begin{proof}
By Remark \ref{rmk-iota} we deduce that
$f\cdot\iota(\xi)=\widetilde{\iota}(r^*f\cdot \widetilde{\xi})$ and $\iota(\xi)\cdot f=\widetilde{\iota}( s^*f\cdot \widetilde{\xi})$ for all $f\in C_0(\widetilde{X})$ and for all $\xi\in C^*_r(G_{ga}^{[0,1)})$.
Since at the parameter $0$ the range and the source map coincide, we have that  $r^*f=s^*f$ and then
 that $[\iota(\xi),f]$ is compact for all $\xi\in C^*_r(G_{ga}^{[0,1)})$ and $f\in C_0(\widetilde{X})$.

Finally, observe that image of 
$C^*_r(G_{ga}^{[0,1)})$ into $\mathbb{B}(\mathcal{H})$ is the closure of a *-algebra of  $\Gamma$-equivariant operators  of finite propagation. It follows that  $\iota(C^*_r(G_{ga}^{[0,1)}))$ is contained in $D^*(\widetilde{X})^\Gamma$.
\end{proof}

As a consequence of the previous lemma we have the following  commutative diagram of C*-algebras.
\begin{equation}\label{iota}
\xymatrix{0\ar[r]&C_r^*(G)\otimes\KK\ar[r]\ar[d]^\iota& C^*_r(G_{ga}^{[0,1)})\ar[r]\ar[d] ^\iota& C^*_r(TX\rtimes\RR^*_+)\ar[r]\ar[d]^\iota&0\\
	0\ar[r]&C^*(\widetilde{X})^\Gamma\ar[r] & D^*(\widetilde{X})^\Gamma\ar[r]& D^*(\widetilde{X})^\Gamma/C^*(\widetilde{X})^\Gamma\ar[r]&0}
\end{equation}

\begin{theorem}\label{th1}
	The vertical arrows of diagram \eqref{iota} induce isomorphisms in K-theory.
\end{theorem}

\begin{proof}
Obviously the *-homomorphism $\iota \colon C_r^*(G)\otimes\KK\to C^*(\widetilde{X})^\Gamma$ induces an isomorphism. So if we prove that $[\iota]\colon K_*(C^*_r(TX\rtimes\RR^*_+))\to K_*(D^*(\widetilde{X})^\Gamma/C^*(\widetilde{X})^\Gamma)$ is an isomorphism, thanks to the Five Lemma, we get the wished result.

First of all recall that, by using the isomorphism \eqref{isoD/C}, we can replace $D^*(\widetilde{X})^\Gamma/C^*(\widetilde{X})^\Gamma$ with $D^*(X)/C^*(X)$.
Since Paschke duality is an isomorphism, it follows that
proving that $\mathcal{P}\circ [\iota]\colon K_*(C^*_r(TX\rtimes\RR^*_+))\to KK^{*+1}(C(X),\CC)$ is an isomorphism is equivalent to prove that $[\iota]$ is so.

Recall that Paschke duality is given by the asymptotic morphism 
$\mu^*(u)$ in \eqref{paschke}, hence $\mathcal{P}\circ [\iota]$ is given by the asymptotic morphism
\begin{equation}\label{eqazione1}
(\iota\otimes \mathrm{id}_{C(X)})^*\mu^*(u)\in E^1(C^*_r(TX\rtimes\RR^*_+)\otimes C(X),\CC).
\end{equation}
Observe that, since $C^*_r(TX\rtimes\RR^*_+)\otimes C(X)$ is nuclear, $(\iota\otimes \mathrm{id}_{C(X)})^*\mu^*(u)$ is  an element of $KK^1(C^*_r(TX\rtimes\RR^*_+)\otimes C(X),\CC)$.
Moreover $(\iota\otimes \mathrm{id}_{C(X)})^*\mu^*(u)=(\mu\circ(\iota\otimes \mathrm{id}_{C(X)}))^*(u)$ and $\mu\circ(\iota\otimes \mathrm{id}_{C(X)})=\iota\circ\bar{\mu}$, where 
\[
\bar{\mu}\colon C^*_r(TX\rtimes\RR^*_+)\otimes C(X)\to C^*_r(TX\rtimes\RR^*_+)
\]
is the  *-homomorphism of C*-algebras given by $\xi\otimes f\mapsto \xi\cdot  r^*f$ (notice that $\xi\cdot r^*f=s^*f\cdot\xi$ so that $\bar{\mu}$ is well defined).

Hence $\mathcal{P}\circ [\iota]$ is given by the KK-element
$\bar{\mu}^*\iota^*(u)$. But $\iota^*(u)$ is exactly the boundary map of the second row of \eqref{TC} for $G= X\times X$ and then for $C^*_r(G)\cong  \KK(L^2(X))$.
So $TC\circ\iota^*(u)\circ TC^{-1}$ is equal to the composition of the KK-element $\partial_X$ in \eqref{dirac} and the suspension isomorphism $S$. Moreover, since $C_0(0,1)\rtimes\RR_+\cong\KK$, the composition of the suspension isomorphism and $TC$ corresponds to the Morita equivalence $KK(\CC, A)\cong KK(\CC, A\otimes \KK)$.

Finally, observe that $TC\circ[\bar{\mu}]\circ TC^{-1}$ is equal to $[m]$, the morphism used in Section \ref{pncr} to define the KK-element $D_X$.
It follows that $TC\circ \mathcal{P}\circ [\iota]\circ TC^{-1}\circ S^{-1}$ is equal to $D_X$, which defines the Poincaré duality of \eqref{pduality}.
In conclusion we have that
\begin{equation}
[\iota]= \mathcal{P}^{-1}\circ TC^{-1}\circ D_X
\end{equation}
and consequently that $[\iota]$ is an isomorphism.

\end{proof}

\subsection{Second approach}\label{second}

Let $D^*(\widetilde{X})^\Gamma$ be the structure Roe algebra associated to the very ample $\Gamma$-equivariant $C_0(\widetilde{X})$-module $L^2(\widetilde{X})\otimes H$, with $H=l^2(\NN)$.

Consider the subalgebra  $L^\infty(X)\otimes\BB(H)\cong L^\infty(\widetilde{X})^\Gamma\otimes\BB(H)\subset \BB(L^2(\widetilde{X})\otimes H)$: 
it is immediate to see that $L^\infty(X)\otimes\BB(H)$ is contained in $D^*(\widetilde{X})^\Gamma$. 

\begin{lemma}\label{isomc}
Let $j$ be the inclusion  $L^\infty(X)\otimes\BB(H)\to D^*(\widetilde{X})^\Gamma$ and let $SD^*(\widetilde{X})$ denote the suspension of $D^*(\widetilde{X})^\Gamma$, then there is an isomorphism
	\[
	K_*(SD^*(\widetilde{X})^\Gamma)\to K_*(\mathcal{C}_j)
	\]
	where $\mathcal{C}_j$ is the mapping cone C*-algebra of $j$.
\end{lemma}
\begin{proof}
	Consider the following exact sequence
	\[
	\xymatrix{0\ar[r]&SD^*(\widetilde{X})^\Gamma\ar[r]&\mathcal{C}_j\ar[r]&L^{\infty}(X)\otimes \BB(H)\ar[r]&0}.
	\]
	By the K\"{u}nneth Theorem, since $K_*(\BB(H))$ is trivial, the K-theory of $L^\infty(X)\otimes \BB(H)$ is trivial too.
	Then the desired isomorphism follows.
\end{proof}

\begin{theorem}
	The following zig-zag induces an isomorphism in K-theory
	\begin{equation}\label{algebras}
	\xymatrix{SD^*(\widetilde{X})^\Gamma\ar[r]&\mathcal{C}_j& \mathcal{C}_{\mathfrak{m}}\ar[l]\ar[r]&\Psi^0(G_{ad}^{[0,1)})& C^*_r(G_{ad}^{[0,1)})\ar[l]}
	\end{equation}
\end{theorem}

\begin{proof}
	The first arrow induces the isomorphism stated in Lemma \ref{isomc},
	the third arrow induces the isomorphism of Lemma \ref{isopdo} and the last arrow gives the isomorphism in \eqref{isoad}.
	The only isomorphism to check is the one induced by the second arrow. The following diagram
	
		\begin{equation}\label{algebras}
		\xymatrix{0\ar[d]&0\ar[d]& 0\ar[d]&0\ar[d]&0\ar[d]\\
			SC^*(\widetilde{X})^\Gamma\ar[r]\ar[d]&SC^*(\widetilde{X})^\Gamma\ar[d]& SC^*_r(G)\ar[l]\ar[r]\ar[d]&SC^*_r(G\ar[d])& SC^*_r(G)\ar[l]\ar[d]\\
			SD^*(\widetilde{X})^\Gamma\ar[r]\ar[d]&\mathcal{C}_j\ar[d]& \mathcal{C}_{\mathfrak{m}}\ar[l]\ar[r]\ar[d]&\Psi^0(G_{ad}^{[0,1)})\ar[d]& C^*_r(G_{ad}^{[0,1)})\ar[l]\ar[d]\\
			S(D^*(X)/C^*(X))\ar[r]\ar[d]&\mathcal{C}_{j'}\ar[d]& \mathcal{C}_{\pi^*}\ar[l]\ar[r]\ar[d]&\Psi^0(G_{ad}^{[0,1)})/SC^*_r(G)\ar[d]& C^*_r(TX)\ar[l]\ar[d]\\
			0&0&0&0&0
			}
		\end{equation}
	is commutative with exact columns and, using the Five lemma, one can prove that all the horizontal arrows but $\mathcal{C}_{\mathfrak{m}}\to \mathcal{C}_{j}$ and $\mathcal{C}_{\pi^*}\to \mathcal{C}_{j'}$ induce isomorphisms in K-theory.
	
	Here $\mathcal{C}_{\pi^*}$ is the mapping cone of $\pi^*\colon C(X)\to C(S^*X)$ and $ \mathcal{C}_{j'}$ is the mapping cone of $j'\colon L^\infty(X)\otimes B(H)\to D^*(X)/C^*(X) $. Notice that here we freely identify $D^*(\widetilde{X})^\Gamma/C^*(\widetilde{X})^\Gamma$ with $D^*(X)/C^*(X) $.
	
In order to apply the Five lemma for the second and the third columns, we are proving that $\mathcal{C}_{\pi^*}\to \mathcal{C}_{j'}$ induces an isomorphism.  To that aim we are going to use, as in the proof of Theorem \ref{th1}, the naturality of Paschke and Poincar\'{e} duality.
 
 Continuous functions on $X$ are multipliers of any algebra among those ones in the third row of diagram \eqref{algebras}. Let $A$ denote any of them, then 
 $C(X)$ commutes with $A$ inside the multipliers and we have a well defined *-homomorphism
 $m_A\colon C(X)\otimes A\to A$.
 
 Moreover observe that $A$ is also the third member in the non-equivariant exact sequences analogous to the ones we are considering, that is the ones  with the ideal equal to the suspension of compact operators $\KK$.
 Let $\partial_A\colon E(A,\CC)$ be the associated boundary map in $E$-theory.  In this case we have no shift because the ideal is the suspension of $\KK$.
 
 So $m_A^*(\partial_A)$ gives a group morphism $K_*(A)\to KK^{*}(C(X), \CC)$ such that
 \begin{equation}\label{triangle}
 \xymatrix{K_*(A)\ar[rr]\ar[rd]_{m_A^*(\partial_A)}& &K_*(A')\ar[ld]^{m_{A'}^*(\partial_{A'})}\\  &KK^{*}(C(X), \CC)&}
 \end{equation}
 is commutative, here the horizontal arrow is induced by the corresponding one among the arrows in the lower row of diagram \eqref{algebras}.
Since for $A$ equal to $	S(D^*(X)/C^*(X))$ and $C^*_r(TX)$, $m_A^*(\partial_A)$ is the isomorphism given by  the Paschke duality and the Poincar\'{e} duality respectively, one proves that $m_A^*(\partial_A)$ is an isomorphism for any $A$, using the commutativity \eqref{triangle} repeatedly.
Now using this fact and the commutativity of the triangle
\eqref{triangle} for $A=\mathcal{C}_{\pi^*}$ and $A'=\mathcal{C}_{j'}$, we deduce that the horizontal arrow 
$\mathcal{C}_{\pi^*}\to \mathcal{C}_{j'}$ induces an isomorphism in K-theory. By the Five Lemma also $\mathcal{C}_{\mathfrak{m}}\to \mathcal{C}_{j}$ induces an isomorphism.
\end{proof}
\section{Stratified spaces}\label{stratified-spaces}
In this section we are going to see that the previous results applies without much more effort to the context of smoothly stratified spaces. 
For the comfort of the reader, it seems to us suitable to treat first the non-singular case and then to explain why it works in the same way for the singular context. This allow to separate the difficulties of the proof (which is the same in both the settings) from the issues arising when one treats stratified spaces.

\subsection{Blow-up groupoid}
We quickly recall
 the blow-up construction in the
groupoid context from \cite{DSblup}. Let $Y$ be a smooth compact manifold and let $X$ be a submanifold of $Y$ and let $DNC(Y,X)$ be the associated deformation to the normal cone, see Section \ref{DNC}.
The group $\RR^*$ acts on $DNC(Y,X)$ by 
$$\lambda\cdot ((x,\xi),0)=((x,\lambda^{-1}\xi),0) \,,\quad 
\lambda\cdot (y,t)=(y,\lambda t) \;\;\text{with}\;\;(x,\xi)\in N_X^Y\,,(y,t)\in Y\times \RR^*\,.
$$
Given a commutative diagram of smooth maps 
\[
\xymatrix{X\ar@{^{(}->}[r]\ar[d]^f& Y\ar[d]^f\\
	X^\prime\ar@{^{(}->}[r]& Y^\prime}
\]
where the horizontal arrows are inclusions of submanifolds, we naturally obtain a smooth map
$DNC(f)\colon DNC(Y,X) \to DNC(Y',X')$. 

This map is defined by $DNC(f)(y,\lambda) = (f (y),\lambda)$
for $y\in Y$ and $\lambda\in\RR$ and $DNC(f)(x,\xi,0) = (f(x),f_N(\xi),0)$ for $(x,\xi) \in N_X^Y$, where $f_N \colon N_X^Y \to N_{X'}^{Y'}$ is the linear map induced by the differential $df$. Moreover it is  equivariant with respect to the action of $\RR^*$.

The action of $\RR^*$ is free and locally proper on $DNC(Y,X)\setminus X\times\RR$ and we define $Blup(Y,X)$ as the quotient space of this action.

If $H\rightrightarrows H^{(0)}$ is a closed subgroupoid of a Lie groupoid $G\rightrightarrows G^{(0)}$, then 
$DNC(G,H)$ is a Lie groupoid over $DNC(G^{(0)},H^{(0)})$ where the source and target maps are simply given by
$DNC (s)$ and $DNC(r)$ as defined above. On the other hand,  $Blup(G,H)$ is not a Lie groupoid over $Blup(G^{(0)},H^{(0)})$, since the $Blup$ construction is not functorial.

\begin{definition}
	The blow-up groupoid of $H$ in $G$ is defined as  the dense open subset of
	$Blup(G,H)$ given by $$Blup_{r,s}(G,H):=\left(DNC(G,H)\setminus (H\times\RR\cup DNC(s)^{-1}(H^{(0)}\times\RR)\cup DNC(r)^{-1}(H^{(0)}\times\RR))\right)/ \RR^*$$
	it is a Lie groupoid over $Blup(G^{(0)},H^{(0)})$.
\end{definition}
We shall be also interested in a variant of this construction: we consider $DNC(G,H)\rightrightarrows  DNC(G^{(0)},H^{(0)})$ 
and define $DNC^+ (G,H)$ as its restriction to $(N^{G^{(0)}}_{H^{(0)}})^+ \times \{0\} \cup G^{(0)}\times \RR^*_+$
with $(N^{G^{(0)}}_{H^{(0)}})^+$ denoting the positive normal bundle, where, for $h\in H^{(0)}$,
	$\left(N^{G^{(0)}}_{H^{(0)}}\right)^+_h$ is defined by $(\RR^n)_+:= \RR^n_+$ once we fix a linear isomorphism 
	$\left(N^{G^{(0)}}_{H^{(0)}}\right)_h$ with  $\RR^n$.
We also define $Blup^+(G,H)$ as the quotient of $DNC^+ (G,H)\setminus H\times \RR_+$ by the action 
of $\RR^*_+$. We obtain in this way the groupoid $$Blup^+_{r,s}(G,H) \rightrightarrows 
Blup^+ (G^{(0)},H^{(0)}).$$

\subsection{Manifolds with fibered corners and iterated edge metrics}
Let us recall the notion of a manifold with fibered corners, due to Melrose.
\begin{definition} Let $X$ be a compact manifold with corners and
	$H_1, . . . ,H_k$ an exhaustive list of its set of boundary hypersurfaces $M_1X$.
	Suppose that each boundary hypersurface $H_i$ is the total space of a smooth
	fibration $\phi_i \colon H_i\to S_i$ where the base $S_i$ is also a compact manifold with corners. The collection of fibrations
	$\phi=(\phi_1,\dots,\phi_k)$ 
	is said to be an iterated
	fibration structure if there is a partial order on the set of hypersurfaces such
	that
	\begin{enumerate}
		\item for any subset $I\subset \{1, \dots, k\}$ with 
		$\bigcap_{i\in I}
		H_i \neq\emptyset$, the set $\{H_i | i \in I\}$
		is totally ordered.
		\item If $H_i < H_j$ , then $H_i \cap H_j  \neq\emptyset$, $\phi_i \colon H_i \cap H_j \to S_i$ is a surjective
		submersion and $S_{ji}:=\phi_j (H_i\cap H_j) \subset S_j$ is a boundary hypersurface
		of the manifold with corners $S_j$ . Moreover, there is a surjective
		submersion $\phi_{ji}\colon S_{ji} \to S_i$ such that on $H_i \cap H_j$ we have
		$\phi_{ji}\circ \phi_j=\phi_i$.
		\item The boundary hypersurfaces of $S_j$ are exactly the $S_ji$ with $H_i < H_j$ .
		In particular if $H_i$ is minimal, then $S_i$ is a closed manifold.
	\end{enumerate}
	Denote by $Z_j$ the fiber of the fibration $\phi_j\colon H_j\to S_j$.
 \end{definition}

The quotient space ${}^{\mathrm{S}}X=X/\sim$, where
\[
x\sim y\Longleftrightarrow\, x=y \quad \mbox{or} \quad \,\exists i\quad \mbox{s.t.}\quad x,y\in H_i \quad\mbox{with}\quad \varphi_i(x)= \varphi_i(y),
\]
is a so-called Thom-Mather stratified space with strata $\{S_1,\dots, S_k\}$, see \cite{Mather}. In turn $X$ is called a resolution of ${}^{\mathrm{S}}X$.

Recall from \cite{ALMP1,ALMP2, ALMP3} that an iterated incomplete iterated edge metric $g$ (shortly an iie-metric) is a metric on $\mathring{X}$ such that in a collar neighbourhood of an hypersurface $H_i$ it takes the form 
\[
dx_i^2+x_i^2g_{Z_i}+ \varphi^*g_{S_i}
\]
where $x_i$ is a boundary defining function of $H_i$ and $g_{Z_j}$ and $g_{S_j}$ are metric with the same structure on $Z_j$ and $S_j$.
In particular  an  iie-metric on $\mathring{X}$
includes a Riemannian metric on each
stratum of ${}^{\mathrm{S}}X$
and that these metrics fit together continuously.

In particular, by \cite[Theorem 2.4.7]{Pflaum} the topology on ${}^{\mathrm{S}}X$
is
that of the metric space with distance between two points given by taking the minimum over
rectificable curves joining them. As a metric space, ${}^{\mathrm{S}}X$
is complete and locally compact \cite[Theorem 2.4.17]{Pflaum} and hence a length space.
\begin{remark}\label{coverings}
Consider a Galois $\Gamma$-covering ${}^{\mathrm{S}}\widetilde{X}$ of ${}^{\mathrm{S}}X$ and its resolution $\widetilde{X}$.
They come with Galois $\Gamma$-coverings $\widetilde{H}_i$ and $\widetilde{S}_i$ over $H_i$ and $S_i$ respectively for all $i$. Moreover we have a $\Gamma$-equivariant lift $\widetilde{\varphi}_i\colon \widetilde{H}_i\to\widetilde{S}_i$ of $\varphi_i$ such that the links are still $Z_i$, see for instance \cite[Section 2.4]{PZ}. Finally from a iie-metric $g$ on ${}^{\mathrm{S}}X$ we obtain a $\Gamma$-equivariant iie-metric $\widetilde{g}$ on ${}^{\mathrm{S}}\widetilde{X}$.
\end{remark}
As in \cite[Section 3.5]{AP}, we can consider the analytic surgery sequence of Higson and Roe for Thom-Mather spaces, 
induced by the following exact sequence of C*-algebras\begin{equation}\label{ASESSing}
\xymatrix{\cdots\ar[r] & K_*(C^*({}^{\mathrm{S}}\widetilde{X})^\Gamma)\ar[r]&  K_*(D^*({}^{\mathrm{S}}\widetilde{X})^\Gamma)\ar[r] &  K_*(D^*({}^{\mathrm{S}}\widetilde{X})^\Gamma /C^*({}^{\mathrm{S}}\widetilde{X})^\Gamma)\ar[r] & \cdots},
\end{equation}
and, as before, we have that  $K_*(D^*({}^{\mathrm{S}}\widetilde{X})^\Gamma /C^*({}^{\mathrm{S}}\widetilde{X})^\Gamma)\simeq KK_{*+1}(C({}^{\mathrm{S}}X), \CC)$, by Paschke duality.
\subsection{Poincar\'{e} duality for stratified spaces}

We can associate a Lie groupoid to a manifold with fibered corners in the following way.
Let $\{H_1,\dots ,H_k\}$ be a list such that if $H_i<H_j\Longrightarrow i<j$ and
observe that if $H_i<H_j$, then $H_i\times_{S_i}H_i\rightrightarrows H_i$ is a closed Lie subgroupoid of 
         $H_j\times_{S_j} H_j$.
\begin{definition}         
Let $G(X,\varphi)\rightrightarrows X$ be the Lie groupoid
\[
Blup^+_{r,s}(\dots(Blup^+_{r,s}(Blup^+_{r,s}(X\times X,H_1\times_{S_1}H_1), H_2\times_{S_2}H_2)\dots, H_k\times_{S_k}H_k)
\]
obtained by a sequence of blowing-up procedures.
\end{definition}
Notice that in this definition the order of the blow-ups is not secondary: if $H_i<H_j$, then there is no immersion of $H_i\times_{S_i} H_i$ into the blow-up of $H_j\times_{S_j} H_j$ into $X\times X$.
As a set $G(X,\varphi)$ is given by 
$$
\mathring{X}\times\mathring{X}\cup \bigsqcup_{j=0}^k (H_j\times_{S_j}TS_j\times_{S_j}H_j)_{|X_j}
$$
where $X_j= H_j\setminus \left(H_j\cap\,\bigcup_{i>j}H_i\right)$.

\begin{definition}
Consider the adiabatic deformation groupoid $G(X,\varphi)_{ad}^{[0,1)}\rightrightarrows X\times [0,1)$. 
Set $X_\partial:= \mathring{X}\cup \left(\partial X\times[0,1) \right)$   and 
define the non-commutative tangent bundle of $X$ as the Lie groupoid
\[
T^{NC}_\varphi X:= \left(G(X,\varphi)_{ad}^{[0,1)}\right)_{|X_\partial}\rightrightarrows X_\partial.
\]
As a set $T^{NC}_\varphi X$ is equal to $T\mathring{X}\cup \bigsqcup_{j=0}^k \left((H_j\times_{S_j}TS_j\times_{S_j}H_j)_{|X_j}\right)_{ad}^{[0,1)}$.

\end{definition}

We thus obtain an exact sequence of C*-algebras analogous to \eqref{AES}
\begin{equation}\label{AESsing}
\xymatrix{0\ar[r]& C^*_r(\mathring{X}\times\mathring{X}\times(0,1))\ar[r]& C^*_r(G(X,\varphi)_{ad}^{[0,1)})\ar[r] &C^*_r(T^{NC}_\varphi X)\ar[r]& 0}.
\end{equation}
Denote by $\partial_{(X,\varphi)}\colon KK(\CC,C^*_r(T^{NC}_\varphi X))\to KK(\CC,C^*_r(\mathring{X}\times\mathring{X}))$ the boundary map associated to \eqref{AESsing}, up to suspension isomorphism.
\begin{theorem}[Poincar\'{e} duality \cite{DLduality,DLR}]\label{pdsing}
	Let ${}^{S}X$ be a Thom-Mather stratified space, then there exists a KK-equivalence 
	$$ KK(\CC,C^*_r(T^{NC}_\varphi X))\to KK(C({}^{S}X),\CC).$$
\end{theorem}
Denote by $q\colon X_{\partial}\to {}^{\mathrm{S}}X$ the obvious quotient map.
	Recall that for a Lie groupoid $G$, the algebra $C(G^{(0)})$ is in the center of the multiplier algebra of $C^*_r(G)$.
	Let ${}^{\mathrm{S}}m\colon C^*_r(T^{NC}_\varphi X)\otimes C({}^{S}X)\to C^*_r(T^{NC}_\varphi X) $ be the well-defined morphism 
	$\xi\otimes f\mapsto \xi\cdot q^*f$.
	Then the \emph{Dirac element}
	$${}^{\mathrm{S}}D_X:=[{}^{\mathrm{S}}m]\otimes_{\tiny{ C^*_r(T^{NC}_\varphi X)}}\partial_{(X,\varphi)}\in KK( C^*_r(T^{NC}_\varphi X)\otimes C({}^{S}X), \CC)$$
	implements, by Kasparov product, the  Poincar\'{e} duality. 
\begin{remark}
Observe that  the Lie algebroid $\mathfrak{A}G(X,\varphi)$ of $G(X,\varphi)$ is non-canonically isomorphic to $TX$ and the anchor map $\mathfrak{A}G(X,\varphi)\to TX$ is an isomorphism over $\mathring{X}$ and it is the projection onto the kernel of $d\varphi_i$ over $H_i$. In particular we have that the continuous sections of $\mathfrak{A}G(X,\varphi)$ are given by the Lie subalgebra of vector fields  over $X$
$$\mathcal{V}_{{\rm e}} (X)=\{\xi\in \mathcal{V}_b (X)\;\;\xi |_{H_i}\,,\;\;\text{is tangent to the fibers of }\;
\varphi_i:H_i\to S_i  \,\;\forall i\}$$
where 
$$\mathcal{V}_b (X) =\{\xi \in C^\infty(X,TX)\,\;\; ; \;\;\xi x_i\in x_i C^\infty (X) \,\forall i\}.$$
In particular a continuous metric $g_e$ for $\mathfrak{A}G(X,\varphi)$ is given by a  so-called iterated edge metric,
which is defined as $\rho^2 g$, where $g$ is iie-metric and the conformal factor is given by $\rho\in C^\infty(X)$, the product of all the boundary defining functions $x_i$, with $i\in \{1,\dots, k\}$. 

It is worth to point out that the proof of Poincar\'{e} duality in \cite{DLR} takes place in the context of \emph{iterated fibered corners} metrics which are associated to a slightly different Lie groupoid: as a set is the same, whereas the smooth structure is different. Neverthless the proof of Theorem \ref{pdsing} goes exactly in the same way if we use iterated edge metrics. In particular observe that the  proof of Poincar\'{e} duality  in \cite{DLduality}, which corresponds up to KK-equivalence to the one in \cite{DLR}, does not depend on the metric we choose to put on ${}^{S}X$.
\end{remark}
\subsection{The main theorem: the stratified case    }

Let $\widetilde{X}$ be as in Remark \ref{coverings} and let us denote $\widetilde{X}\setminus\partial\widetilde{X}$ by $\widetilde{X}^{reg}$.
Consider the Lie groupoid  $G(\widetilde{X},\widetilde{\varphi})\rightrightarrows\widetilde{X}$  given by 
$$
\widetilde{X}^{reg}\times\widetilde{X}^{reg}\cup \bigsqcup_{j=0}^k (\widetilde{H}_j\times_{\widetilde{S}_j}T\widetilde{S}_j\times_{\widetilde{S}_j}\widetilde{H}_j)_{|\widetilde{X}_j}
$$
and observe that there is a proper and free action of $\Gamma$ on $G(\widetilde{X},\widetilde{\varphi})$ through groupoid automorphisms:  let $g$ be an element of $\Gamma$ and $(x,y)$ is in  $\widetilde{X}^{reg} \times\widetilde{X}^{reg}$ , then $g\cdot(x,y)=(g\cdot x,g\cdot y)$; if instead $(x,\xi,y,t)$ is an element over the boundary, then $g\cdot(x,\xi,y,t)= (g\cdot x,dg\cdot\xi,g\cdot y,t)$. 

\begin{definition}
	Define the groupoid $\mathcal{G}\rightrightarrows X$ as the quotient of  $G(\widetilde{X},\widetilde{\varphi})\rightrightarrows\widetilde{X}$  by the action of $\Gamma$.
	As a set, it is given by 
	$$
	\widetilde{X}^{reg}\times_{\Gamma}\widetilde{X}^{reg}\cup\bigsqcup_{j=0}^k (H_j\times_{S_j}TS_j\times_{S_j}H_j)_{|X_j}.
	$$
\end{definition}

Consequently we have the following exact sequence of C*-algebras
$$
\xymatrix{0\ar[r]& C^*_r(	\widetilde{X}^{reg}\times_{\Gamma}\widetilde{X}^{reg}\times(0,1))\ar[r]& C^*_r(\mathcal{G}_{ad}^{[0,1)})\ar[r] &C^*_r(T^{NC}_\varphi X)\ar[r]& 0}.
$$
Notice that $C^*_r(	\widetilde{X}^{reg}\times_{\Gamma}\widetilde{X}^{reg}\times(0,1)\rtimes\RR_+)$ is a subalgebra of
$\mathcal{H}':= L^2(\widetilde{X}^{reg}\times(0,1),g')$ where we endow $\widetilde{X}$ with a complete iterated edge metric $g'$.
Through the multiplication by the total boundary function $\rho$ we get an isomorphism $m(\rho)\colon \mathcal{H}\to \mathcal{H}'$ with $\mathcal{H}:= L^2(\widetilde{X}^{reg}\times(0,1),g)$, where $g:= \rho^{-2}g'$ is a iie-metric. 
It is a $\Gamma$-equivariant $C_0({}^{S}\widetilde{X})$-module and one can immediately see that the conjugation by $m(\rho)$ maps
$C^*_r(	\widetilde{X}^{reg}\times_{\Gamma}\widetilde{X}^{reg}\times(0,1)\rtimes\RR_+)$  isomorphically onto  $C^*({}^{S}\widetilde{X})^\Gamma$. As in Section \ref{first} this isomorphism extends to an injective map ${}^{S}\iota\colon C^*_r(\mathcal{G}_{ga}^{[0,1)})\to \mathbb{B}(\mathcal{H})$.
\begin{lemma}
	The image of ${}^{S}\iota$ is contained in $D^*({}^{S}\widetilde{X})^\Gamma$.
\end{lemma}
\begin{proof}
	The proof is similar to the one of Lemma \ref{mappa}. The only additional thing to point out is that the commutator of 
	$f\in C_0({}^{S}\widetilde{X})$ and an element of ${}^{S}\iota\left(C^*_r(\mathcal{G}_{ga}^{[0,1)})\right)$ is zero on the singular part, which is a necessary condition for being locally compact.
Recall, from the discussion in Remark \ref{rmk-iota}, that we can see $C^*_r(\mathcal{G}_{ga}^{[0,1)})$ as a subalgebra in the multiplier algebra of $C^*_r(G(\widetilde{X},\widetilde{\varphi}))$ generated by properly supported $\Gamma$-equivariant elements. Let $\widetilde{q}\circ pr\colon \widetilde{X}\times[0,1)\to {}^{S}\widetilde{X}$ the composition of the projection to $\widetilde{X}$ and the quotient map. Then
it follows that $\widetilde{q}^*f$ is constant along the fibers of $\varphi_i$ for all $i= 1,\dots, k$ and this implies that $r^*(\widetilde{q}\circ pr)^*f= s^*(\widetilde{q}\circ pr)^*f$ is constant on $G(\widetilde{X},\widetilde{\varphi})_{|\partial \widetilde{X}}$. Consequently $[f,{}^{S}\iota (x)]$ is zero on the singular part of ${}^{S}\widetilde{X}$ and then locally compact.
\end{proof}

Now we are able to state the main result of this section whose
proof follows exactly the proof of Theorem \ref{th1}.
 
 \begin{theorem}\label{main-stratified}
 There exists a commutative diagram 
 	\begin{equation}\label{iotasing}
 	\xymatrix{0\ar[r]&C_r^*(	\widetilde{X}^{reg}\times_{\Gamma}\widetilde{X}^{reg})\otimes\KK\ar[r]\ar[d]^{{}^{S}\iota}& C^*_r(\mathcal{G}_{ga}^{[0,1)})\ar[r]\ar[d] ^{{}^{S}\iota}& C^*_r(T^{NC}_\varphi X\rtimes\RR^*_+)\ar[r]\ar[d]^{{}^{S}\iota}&0\\
 		0\ar[r]&C^*({}^{S}\widetilde{X})^\Gamma\ar[r] & D^*({}^{S}\widetilde{X})^\Gamma\ar[r]& D^*({}^{S}\widetilde{X})^\Gamma/C^*({}^{S}\widetilde{X})^\Gamma\ar[r]&0}
 	\end{equation}
 such that the vertical arrows induce isomorphisms in K-theory.
 \end{theorem}
 
 \begin{remark}
 Let us  highlight that one can also follow the second approach in Section \ref{second}, since  
 Lemma \ref{isopdo} holds also for $\mathcal{G}$.
More precisely we obtain that the analogous of the middle column of \eqref{algebras}  is given by
 \begin{equation}
 \xymatrix{0\ar[r]& C^*_r(\mathcal{G}_{| X^{reg}})\otimes C_0(0,1)\ar[r]& \mathcal{C}(C(X)\xrightarrow{\mathfrak{m}} \Psi^0(\mathcal{G}))\ar[r]& \mathcal{C}(C(X)\xrightarrow{\mathfrak{m}}\Sigma_{nc}(X))\ar[r]&0}
\end{equation}
where $\Sigma_{nc}(X):=\Psi^0(\mathcal{G})/C^*_r(\mathcal{G}_{| X^{reg}})$ is the C*-algebra of non-commutative symbols.
This C*-algebra is given by the following pull-back.
\begin{equation}
\xymatrix{
	\Sigma_{nc}(X)\ar[r]\ar[d]&\Psi^0((\mathcal{G}_{|\partial X})_{ad}^{[0,1)})\ar[d]\\
		C(\mathfrak{S}^*\mathcal{G})\ar[r] &
	C(\mathfrak{S}^*\mathcal{G}_{|\partial X})
	}
\end{equation} 
 	\end{remark}
 	
 \begin{remark}
 	Notice that
 	 Theorem \ref{main-stratified} reveals the correspondence between K-theoretic invariants associated to incomplete metrics and conformally correspondent complete metrics. Indeed  in the first row complete metrics are used to define the C*-algebras, whereas in the second row the metrics are incomplete.
 \end{remark}
\section{Comparing secondary invariants}
In this section we shall prove that the isomorphism $K_{*+1}(D^*(\widetilde{X})^{\Gamma})\cong K_{*}(C^*_r(G^{[0,1)}_{ad}))$, induced by \eqref{algebras},  put in correspondence the $\varrho$-classes associated to  a metric $g$ with positive scalar curvature, defined in \cite{PS1} and \cite{Zadiabatic}. 

Let  $\widetilde{X}$ be a smooth spin manifold with a free, proper and isometric action of $\Gamma$.  Let $\slashed{S}$ denote the spinor bundle over $\widetilde{X}$. Let $g$ be a $\Gamma$-invariant  complete metric on $\widetilde{X}$
and assume that the scalar curvature of $g$ is positive everywhere on $\widetilde{X}$.
The Lichnerowicz formula implies that the dirac operator $\slashed{D}$ associated to $g$ is invertible. 

Denote by $\chi\colon \RR\to \RR$ the sign function and by $\psi\colon \RR\to \RR$ the chopping function $t\mapsto \frac{t}{\sqrt{1+t^2}}$.
There is a continuous path of functions $\psi_s\colon t\mapsto \psi(\frac{t}{1-s})$ such that $\psi_0=\psi$ and $\psi_1=\chi$ (actually it is a continuous path of continuous functions on $\RR\setminus \{0\}$).
\subsection{Coarse invariants}
Let us recall the definition of the  $\varrho$-classes of Piazza and Schick in \cite{PS1}.

\begin{definition}
	Let $\dim(\widetilde{X})$ be odd.
	Since $\slashed{D}$ is invertible, the operator $\chi(\slashed{D})$ is a symmetry in $D^*(\widetilde{X})^\Gamma$. Then we can define $\varrho(g)$
	as the class
	\[
	\left[\frac{1}{2}(\chi(\slashed{D})+1)\right]\in K_0(D^*(\widetilde{X})^\Gamma).
	\]
\end{definition}

Here  $D^*(\widetilde{X})^\Gamma$ is represented on the very ample $\Gamma$-equivariant $C_0(\widetilde{X})$-module $L^2(\slashed{S})\otimes l^2\NN$ and $\frac{1}{2}(\chi(\slashed{D})+1)$ is intended as the infinite matrix with $\frac{1}{2}(\chi(\slashed{D})+1)$ in the top left corner and the identity along all the diagonal.

\begin{remark}
	Notice that, in the odd dimensional case, $\frac{1}{2}(\chi(\slashed{D})+1)$ is exactly $\mathcal{P}_>$, the projection on the positive part of the spectrum of $\slashed{D}$. Consequently $\varrho(g)$ is the image of $[\mathcal{P}_>]$ through the map $K_0(\Psi^0_\Gamma(\widetilde{X}))\to K_0(D^*(\widetilde{X})^\Gamma)$.
\end{remark}

Let us now consider the even dimensional context. In this case the spinor bundle is graded by the chirality element and it splits in the following way $\slashed{S}=\slashed{S}_+\oplus \slashed{S}_-$. In turn the Dirac operator is odd with respect to the grading and it is of the following matrix form $\begin{bmatrix}
0 & \slashed{D}_+\\\slashed{D}_-&0
\end{bmatrix}$. 

Notice that, even though $\slashed{S}_+$ and $\slashed{S}_-$ are not isomorphic as smooth bundles, there exists an isometric $\Gamma$-equivariant isomorphism $u\colon \slashed{S}_-\to\slashed{S}_+$ of measurable bundle, which is given by the Clifford multiplication by any vector field whose zero set is of measure equal to zero. It induces the unitary $\Gamma$-equivariant maps $U\colon L^\infty(\slashed{S}_-)\to L^\infty(\slashed{S}_+)$ and 
$U\colon L^2(\slashed{S}_-)\to L^2(\slashed{S}_+)$.
\begin{definition}\label{even-coarse}
Let $\chi_+(\slashed{D})$ be the positive part of $\chi(\slashed{D})$. Then $\varrho(g)$ is defined by the class
\[
[U\chi_+(\slashed{D})]\in K_1(D^*(\widetilde{X})^\Gamma).
\] 
\end{definition}
Here $D^*(\widetilde{X})^\Gamma$ is represented on the $\Gamma$-equivariant $C_0(\widetilde{X})$-module $L^2(\slashed{S}_+)\otimes l^2\NN$.
Moreover notice that the definition does not depends on the choice of $U$, see \cite[Section 2B2]{PS2}.

\subsection{Adiabatic invariants}
Since $\widetilde{X}$ is spin, the Lie algebroid of the adiabatic deformation of $G=\widetilde{X}\times_\Gamma\widetilde{X}$, which is  $TX\times[0,1]$, is obviously spin. So we can consider $\slashed{D}_{ad}$, the Dirac operator of $G^{[0,1]}_{ad}$,
 defined on the $C_c^\infty(G^{[0,1]}_{ad})$-module 
$C^\infty_c(G^{[0,1]}_{ad},r^*\slashed{S}\otimes\Omega^{\frac{1}{2}})$. Let us denote by $\mathcal{E}^{[0,1]}_{ad}$ its $C^*_r(G^{[0,1]}_{ad})$-completion and let us denote by $\mathcal{E}$ its restriction at $t=1$.

As explained in Section \ref{PDO} we can consider it as a field of operators such that at $t=1$ is the $\Gamma$-equivariant Dirac operator $\slashed{D}$ of $\widetilde{X}$
and at $t=0$ is the given by the Fourier transform of its symbol, namely  by Clifford multiplication.

Notice that $\psi(\slashed{D}_{ad})$ belongs to $\Psi^0(G^{[0,1]}_{ad})$. Moreover, since the restriction of  $\slashed{D}_{ad}$ at $t=1$ is invertible, we have a continuous path of operators $\psi_s(\slashed{D})$ from $\psi(\slashed{D})$ to $\chi(\slashed{D})$.

\begin{definition}
	Define  $\varrho^{ad}(g)$ as the class in $KK^*(\CC,C^*_r(G^{[0,1]}_{ad}))$ given by the concatenation of the Kasparov bimodules $[\mathcal{E}^{[0,1]}_{ad}, \psi(\slashed{D}_{ad})]$ and $[\mathcal{E}\otimes C_0[0,1), \psi_s(\slashed{D})]$, after a suitable reparametrization.
	More precisely the Hilbert module is given by $\mathcal{E}_{ad}^{[0,1)}$, where the notation is self-explanatory, and let us denote by $\psi_{ad}^{[0,1)}(\slashed{D})$ the corresponding operator.
\end{definition}
This definition makes sense in both the odd and the even dimensional case, because the definition of KK-groups take into account the grading of the spinor bundle.

\subsection{Comparison of  $\varrho$-classes}

Let us first calculate the image of 
 of $\varrho^{ad}(g)$ into $KK^*(\CC, \Psi^0(G^{[0,1)}_{ad}))$ through the inclusion $C^*_r(G_{ad}^{[0,1)})\hookrightarrow\Psi^0(G_{ad}^{[0,1)})$. From the definition of $\mathcal{E}_{ad}^{[0,1)}$ it is easy to see that  $$\mathcal{E}_{ad}^{[0,1)}\otimes_{\Psi^0(G_{ad}^{[0,1)})}\Psi^0(G_{ad}^{[0,1)})\cong \Psi^0(G_{ad}^{[0,1)},\slashed{S}),$$ whereas the operator can be seen as unchanged.
 
We are going to treat separately the odd and the even dimensional case.
Let us start with the odd case and let us denote  by $\Lambda$ the isomorphism $K_1(SD^*(\widetilde{X})^\Gamma)\to K_1(C^*_r(G^{[0,1)}_{ad}))$.
Recall that in the odd dimensional case $\varrho(g)$ is the image of $[\mathcal{P}_>]$ through the inclusion of $\Psi^0_\Gamma(\widetilde{X})$ into $D^*(\widetilde{X})^\Gamma$.
Since the following triangle 
\[
\xymatrix{&K_1(S\Psi_\Gamma^0(\widetilde{X}))\ar[dl]\ar[dr]&\\
	           K_1(SD^*(\widetilde{X})^\Gamma)\ar[rr]^\Lambda &&K_1(\Psi^0(G^{[0,1)}_{ad}))}
\]
is obviously commutative, it is enough to compare the image of the suspension of $[\mathcal{P}_>]$ and the image of  $\varrho^{ad}(g)$ inside $K_1(\Psi^0(G^{[0,1)}_{ad}))$.
The suspension of $[\mathcal{P}_>]$ is given by the path of unitaries
$$\exp(2\pi it\mathcal{P}_>)= e^{2\pi i t}\mathcal{P}_>+(1-\mathcal{P}_>)\in C_0(0,1)\otimes \Psi_\Gamma^0(\widetilde{X},\widetilde{\slashed{S}}).$$ 
First observe that the identification 
$KK^1(\CC,A)\cong K_1(A)$
is given by the following map
$$[H,F]\mapsto [\exp(2\pi i P)]$$
 where $P=\frac{F+1}{2}$.
An easy calculation shows that $[\exp(2\pi it\mathcal{P}_>)]$ corresponds to the 
Kasparov bimodule $$[\mathcal{H}, t\chi(\slashed{D})+(t-1)]\in KK^1(\CC,C_0(0,1)\otimes \Psi_\Gamma^0(\widetilde{X}))$$
where $\mathcal{H}$ is the $C_0(0,1)\otimes \Psi_\Gamma^0(\widetilde{X})$-module $C_0(0,1)\otimes \Psi_\Gamma^0(\widetilde{X},\widetilde{\slashed{S}})$. The operator $t\chi(\slashed{D})+(t-1)$ extends to the $\Psi^0(G^{[0,1)}_{ad})$-module 
$\Psi^0(G^{[0,1)}_{ad},\slashed{S})$ and we obtain the corresponding element in $KK(\CC,\Psi^0(G^{[0,1)}_{ad}))$, obtained through the
inclusion $C_0(0,1)\otimes \Psi_\Gamma^0(\widetilde{X})\hookrightarrow\Psi^0(G^{[0,1)}_{ad})$.

Finally observe that 
 $t\chi(\slashed{D})+(t-1)$ and $\psi_{ad}^{[0,1)}(\slashed{D})$, both of them operators on $\Psi^0(G^{[0,1)}_{ad},\slashed{S})$, commute.
 Then by \cite[Lemma 11]{Sk-remarks} there is a homotopy connecting them, hence the images of $\varrho(g)$ and $\varrho^{ad}(g)$ coincide in $KK^1(\CC,\Psi^0(G^{[0,1)}_{ad}))$.
 
 \medskip 
 Let us now pass to the even dimensional case. We refer the reader to  \cite[Section 2.2, 2.3]{AAS} for a detailed account about the isomorphism between the relative K-group of a morphism and the K-theory group of its mapping cone C*-algebra.
 In this case we are going to start with the class of $K_1(D^*(\widetilde{X})^{\Gamma})$ induced by the unitary $U\chi(\slashed{D})_+$ of Definition \ref{even-coarse}.
 Following the arrows in \eqref{algebras}, we see that it induces the class $[L^\infty(\widetilde{X},\widetilde{\slashed{S}}_+),L^\infty(\widetilde{X},\widetilde{\slashed{S}}_+),U\chi(\slashed{D})_+]$ in the relative group $K_0(j)$. Using any path of unitaries from $U$ to the identity and the fact that $U^{-1}L^\infty(\widetilde{X},\widetilde{\slashed{S}}_+)=L^\infty(\widetilde{X},\widetilde{\slashed{S}}_-)$, we see that the last class is equal to $[L^\infty(\widetilde{X},\widetilde{\slashed{S}}_+),L^\infty(\widetilde{X},\widetilde{\slashed{S}}_-),\chi(\slashed{D})_+]$, which in turn is clearly the image of $[C(\widetilde{X},\widetilde{\slashed{S}}_+),C(\widetilde{X},\widetilde{\slashed{S}}_-),\chi(\slashed{D})_+]\in K_0(\mathfrak{m})$ through the second arrows in \eqref{algebras}.
 Now observe that $\mathfrak{m}$ is injective and, as explained in \cite[Section 2]{AAS}, one can easily see that the realization of our class in  $K_0(C_\mathfrak{m})$ is given by 
 \begin{equation}\label{classmc}
 \left[\begin{pmatrix}\cos^2(\frac{\pi}{2}t)1_+&\cos(\frac{\pi}{2}t)\sin(\frac{\pi}{2}t) \chi(\slashed{D})_+\\\cos(\frac{\pi}{2}t)\sin(\frac{\pi}{2}t) \chi(\slashed{D})_-&\sin^2(\frac{\pi}{2}t)1_-\end{pmatrix} \right]  -\left[\begin{pmatrix}0&0\\0&1_-\end{pmatrix} \right]\quad t\in[0,1]
 \end{equation}
where $1_{\pm}$ is the identity of the $\Psi_\Gamma^0(\widetilde{X})$-module $C(\widetilde{X},\widetilde{\slashed{S}}_\pm)\otimes_{C(X)}\Psi^0_\Gamma(\widetilde{X})=\Psi^0_\Gamma(\widetilde{X},\slashed{S}_\pm)$ and the second matrix is meant to denote the constant path.
The first term of \eqref{classmc} is obtained by conjugating $\begin{pmatrix}1_+&0\\0&0\end{pmatrix}$ with the path of invertible matrices 
\begin{equation}\label{inv1} \begin{pmatrix}\cos(\frac{\pi}{2}t)&-\sin(\frac{\pi}{2}t) \chi(\slashed{D})_+\\\sin(\frac{\pi}{2}t) \chi(\slashed{D})_-&\cos(\frac{\pi}{2}t)\end{pmatrix}
\end{equation}
 and
the last path of invertible martices is homotopic, through paths of invertible elements, to 
\begin{equation}\label{inv2} \begin{pmatrix}\cos^2(\frac{\pi}{2}t)&-\sin(\frac{\pi}{2}t) \chi(\slashed{D})_+\\\sin(\frac{\pi}{2}t) \chi(\slashed{D})_-&\cos^2(\frac{\pi}{2}t)\end{pmatrix}.
\end{equation}

Now conjugating $\begin{pmatrix}1_+&0\\0&0\end{pmatrix}$ by \eqref{inv2} instead of \eqref{inv1}, we obtain  exactly the image of $\left[\Psi^0_\Gamma(\widetilde{X},\slashed{S}_+)\oplus\Psi^0_\Gamma(\widetilde{X},\slashed{S}_-),F\right]$  by means of the standard identification of $KK(\CC,C_\mathfrak{m})$ and $K_0(C_\mathfrak{m})$,
here the operator in the Kasparov bimodule is given by $$F=\begin{pmatrix}0& -\sin(\frac{\pi}{2}t)\chi(\slashed{D})_+\\ \sin(\frac{\pi}{2}t)\chi(\slashed{D})_-&0\end{pmatrix}. $$

Finally if we move to $KK(\CC,\Psi^0(G^{[0,1)}_{ad}))$ through the map $\eta$ from Lemma \ref{isopdo}, we obtain the class $[\Psi^0(G_{ad}^{[0,1)},\slashed{S}),F]$
which is, by  \cite[Lemma 11]{Sk-remarks}, operatorially homotopic to the image of $\varrho^{ad}(g)$ in $KK(\CC,\Psi^0(G^{[0,1)}_{ad}))$. 
Observe that this is true because the identity is a compact operator on the module $\Psi^0(G_{ad}^{[0,1]})$. Summarizing we have proved the following result.
\begin{theorem}
	The classes $\varrho(g)$ and $\varrho^{ad}(g)$ correspond to each other through the isomorphism  induced by \eqref{algebras}.
\end{theorem}

\addcontentsline{toc}{section}{References}
\bibliographystyle{plain}
\nocite{*}
\bibliography{CoarseGrpd}

\begin{thebibliography}{10}

\bibitem{ALMP1}
Pierre Albin, \'{E}ric Leichtnam, Rafe Mazzeo, and Paolo Piazza.
\newblock The signature package on {W}itt spaces.
\newblock {\em Ann. Sci. \'{E}c. Norm. Sup\'{e}r. (4)}, 45(2):241--310, 2012.

\bibitem{ALMP2}
Pierre Albin, Eric Leichtnam, Rafe Mazzeo, and Paolo Piazza.
\newblock The {N}ovikov conjecture on {C}heeger spaces.
\newblock {\em J. Noncommut. Geom.}, 11(2):451--506, 2017.

\bibitem{ALMP3}
Pierre Albin, Eric Leichtnam, Rafe Mazzeo, and Paolo Piazza.
\newblock Hodge theory on {C}heeger spaces.
\newblock {\em J. Reine Angew. Math.}, 744:29--102, 2018.

\bibitem{AP}
Pierre Albin and Paolo Piazza.
\newblock Stratified surgery and {K}-theory invariants of the signature
  operator.
\newblock arXiv:1710.00934.

\bibitem{AAS}
Paolo Antonini, Sara Azzali, and Georges Skandalis.
\newblock Flat bundles, von {N}eumann algebras and {$K$}-theory with
  {$\Bbb{R}/\Bbb{Z}$}-coefficients.
\newblock {\em J. K-Theory}, 13(2):275--303, 2014.

\bibitem{BR1}
Moulay-Tahar Benameur and Indrava Roy.
\newblock The {H}igson-{R}oe sequence for \'{e}tale groupoids. {I}. {D}ual
  algebras and compatibility with the {B}{C} map.
\newblock arXiv:1801.06040.

\bibitem{BR2}
Moulay-Tahar Benameur and Indrava Roy.
\newblock The {H}igson-{R}oe sequence for \'{e}tale groupoids. {II}. the
  universal sequence for equivariant families.
\newblock arXiv:1812.04371.

\bibitem{Connes-thom}
A.~Connes.
\newblock An analogue of the {T}hom isomorphism for crossed products of a
  {$C^{\ast} $}-algebra by an action of {${\bf R}$}.
\newblock {\em Adv. in Math.}, 39(1):31--55, 1981.

\bibitem{CS}
A.~Connes and G.~Skandalis.
\newblock The longitudinal index theorem for foliations.
\newblock {\em Publ. Res. Inst. Math. Sci.}, 20(6):1139--1183, 1984.

\bibitem{connes-book}
Alain Connes.
\newblock {\em Noncommutative geometry}.
\newblock Academic Press, Inc., San Diego, CA, 1994.

\bibitem{DLisolated}
Claire Debord and Jean-Marie Lescure.
\newblock {$K$}-duality for pseudomanifolds with isolated singularities.
\newblock {\em J. Funct. Anal.}, 219(1):109--133, 2005.

\bibitem{DLduality}
Claire Debord and Jean-Marie Lescure.
\newblock {$K$}-duality for stratified pseudomanifolds.
\newblock {\em Geom. Topol.}, 13(1):49--86, 2009.

\bibitem{DL}
Claire Debord and Jean-Marie Lescure.
\newblock Index theory and groupoids.
\newblock In {\em Geometric and topological methods for quantum field theory},
  pages 86--158. Cambridge Univ. Press, Cambridge, 2010.

\bibitem{DLR}
Claire Debord, Jean-Marie Lescure, and Fr\'{e}d\'{e}ric Rochon.
\newblock Pseudodifferential operators on manifolds with fibred corners.
\newblock {\em Ann. Inst. Fourier (Grenoble)}, 65(4):1799--1880, 2015.

\bibitem{DSblup}
Claire Debord and Georges Skandalis.
\newblock Blowup constructions for {L}ie groupoids and a {B}outet de {M}onvel
  type calculus.
\newblock arXiv:1705.09588.

\bibitem{DS}
Claire Debord and Georges Skandalis.
\newblock Adiabatic groupoid, crossed product by {$\Bbb{R}_+^\ast$} and
  pseudodifferential calculus.
\newblock {\em Adv. Math.}, 257:66--91, 2014.

\bibitem{DG2}
Robin~J. Deeley and Magnus Goffeng.
\newblock Realizing the analytic surgery group of {H}igson and {R}oe
  geometrically part {II}: relative {$\eta$}-invariants.
\newblock {\em Math. Ann.}, 366(3-4):1319--1363, 2016.

\bibitem{DG3}
Robin~J. Deeley and Magnus Goffeng.
\newblock Realizing the analytic surgery group of {H}igson and {R}oe
  geometrically part {III}: higher invariants.
\newblock {\em Math. Ann.}, 366(3-4):1513--1559, 2016.

\bibitem{DG1}
Robin~J. Deeley and Magnus Goffeng.
\newblock Realizing the analytic surgery group of {H}igson and {R}oe
  geometrically, part {I}: the geometric model.
\newblock {\em J. Homotopy Relat. Struct.}, 12(1):109--142, 2017.

\bibitem{FS}
Thierry Fack and Georges Skandalis.
\newblock Connes' analogue of the {T}hom isomorphism for the {K}asparov groups.
\newblock {\em Invent. Math.}, 64(1):7--14, 1981.

\bibitem{HR1}
Nigel Higson and John Roe.
\newblock Mapping surgery to analysis. {I}. {A}nalytic signatures.
\newblock {\em $K$-Theory}, 33(4):277--299, 2005.

\bibitem{HR2}
Nigel Higson and John Roe.
\newblock Mapping surgery to analysis. {II}. {G}eometric signatures.
\newblock {\em $K$-Theory}, 33(4):301--324, 2005.

\bibitem{HR3}
Nigel Higson and John Roe.
\newblock Mapping surgery to analysis. {III}. {E}xact sequences.
\newblock {\em $K$-Theory}, 33(4):325--346, 2005.

\bibitem{HS}
Michel Hilsum and Georges Skandalis.
\newblock Morphismes {$K$}-orient\'{e}s d'espaces de feuilles et
  fonctorialit\'{e} en th\'{e}orie de {K}asparov (d'apr\`es une conjecture
  d'{A}. {C}onnes).
\newblock {\em Ann. Sci. \'{E}cole Norm. Sup. (4)}, 20(3):325--390, 1987.

\bibitem{Kasparov}
G.~G. Kasparov.
\newblock Equivariant {$KK$}-theory and the {N}ovikov conjecture.
\newblock {\em Invent. Math.}, 91(1):147--201, 1988.

\bibitem{Legall}
Pierre-Yves Le~Gall.
\newblock Th\'{e}orie de {K}asparov \'{e}quivariante et groupo\"{i}des. {I}.
\newblock {\em $K$-Theory}, 16(4):361--390, 1999.

\bibitem{Mather}
John Mather.
\newblock Notes on topological stability.
\newblock {\em Bull. Amer. Math. Soc. (N.S.)}, 49(4):475--506, 2012.

\bibitem{MP}
Bertrand Monthubert and Fran\c{c}ois Pierrot.
\newblock Indice analytique et groupo\"{i}des de {L}ie.
\newblock {\em C. R. Acad. Sci. Paris S\'{e}r. I Math.}, 325(2):193--198, 1997.

\bibitem{NWX}
Victor Nistor, Alan Weinstein, and Ping Xu.
\newblock Pseudodifferential operators on differential groupoids.
\newblock {\em Pacific J. Math.}, 189(1):117--152, 1999.

\bibitem{Pflaum}
Markus~J. Pflaum.
\newblock {\em Analytic and geometric study of stratified spaces}, volume 1768
  of {\em Lecture Notes in Mathematics}.
\newblock Springer-Verlag, Berlin, 2001.

\bibitem{PS1}
Paolo Piazza and Thomas Schick.
\newblock Rho-classes, index theory and {S}tolz' positive scalar curvature
  sequence.
\newblock {\em J. Topol.}, 7(4):965--1004, 2014.

\bibitem{PS2}
Paolo Piazza and Thomas Schick.
\newblock The surgery exact sequence, {K}-theory and the signature operator.
\newblock {\em Ann. K-Theory}, 1(2):109--154, 2016.

\bibitem{PZ}
Paolo Piazza and Vito~Felice Zenobi.
\newblock Singular spaces, groupoids and metrics of positive scalar curvature.
\newblock {\em J. Geom. Phys.}, 137:87--123, 2019.

\bibitem{QR}
Yu~Qiao and John Roe.
\newblock On the localization algebra of {G}uoliang {Y}u.
\newblock {\em Forum Math.}, 22(4):657--665, 2010.

\bibitem{Roeassembly}
John Roe.
\newblock Comparing analytic assembly maps.
\newblock {\em Q. J. Math.}, 53(2):241--248, 2002.

\bibitem{Sk-remarks}
Georges Skandalis.
\newblock Some remarks on {K}asparov theory.
\newblock {\em J. Funct. Anal.}, 56(3):337--347, 1984.

\bibitem{JLT}
Jean~Louis Tu.
\newblock La conjecture de {N}ovikov pour les feuilletages hyperboliques.
\newblock {\em $K$-Theory}, 16(2):129--184, 1999.

\bibitem{Vassout}
St\'{e}phane Vassout.
\newblock Unbounded pseudodifferential calculus on {L}ie groupoids.
\newblock {\em J. Funct. Anal.}, 236(1):161--200, 2006.

\bibitem{WXY}
Shmuel Weinberger, Zhizhang Xie, and Guoliang Yu.
\newblock Additivity of higher rho invariants and nonrigidity of topological
  manifolds.
\newblock Preprint 2016. arXiv 1608.03661.

\bibitem{XY1}
Zhizhang Xie and Guoliang Yu.
\newblock Positive scalar curvature, higher rho invariants and localization
  algebras.
\newblock {\em Adv. Math.}, 262:823--866, 2014.

\bibitem{XY2}
Zhizhang Xie and Guoliang Yu.
\newblock Higher rho invariants and the moduli space of positive scalar
  curvature metrics.
\newblock {\em Adv. Math.}, 307:1046--1069, 2017.

\bibitem{Yu}
Guoliang Yu.
\newblock Localization algebras and the coarse {B}aum-{C}onnes conjecture.
\newblock {\em $K$-Theory}, 11(4):307--318, 1997.

\bibitem{Z}
Rudolf Zeidler.
\newblock Positive scalar curvature and product formulas for secondary index
  invariants.
\newblock {\em J. Topol.}, 9(3):687--724, 2016.

\bibitem{Zadiabatic}
Vito~Felice Zenobi.
\newblock Adiabatic groupoids and secondary invariants in {K}-theory.
\newblock arXiv:1609.08015.

\bibitem{zenobi}
Vito~Felice Zenobi.
\newblock Mapping the surgery exact sequence for topological manifolds to
  analysis.
\newblock {\em J. Topol. Anal.}, 9(2):329--361, 2017.

\end{thebibliography}

\end{document}